\documentclass[12pt, twoside]{article}
\usepackage{amsmath,amsthm,amssymb}
\usepackage{times}
\usepackage{enumerate}
\usepackage{bm}
\usepackage[all]{xy}
\usepackage{mathrsfs}
\usepackage{amscd}
\usepackage{color}
\def\titlerunning#1{\gdef\titrun{#1}}
\makeatletter
\def\author#1{\gdef\autrun{\def\and{\unskip, }#1}\gdef\@author{#1}}
\def\address#1{{\def\and{\\\hspace*{18pt}}\renewcommand{\thefootnote}{}%
		\footnote {#1}}%
	\markboth{\autrun}{\titrun}}
\makeatother
\def\email#1{e-mail: #1}
\def\subjclass#1{{\renewcommand{\thefootnote}{}%
		\footnote{\emph{Mathematics Subject Classification (2020):} #1}}}
\def\keywords#1{\par\medskip
	\noindent\textbf{Keywords.} #1}

\newtheorem{theorem}{Theorem}[section]
\newtheorem{corollary}[theorem]{Corollary}
\newtheorem{lemma}[theorem]{Lemma}
\newtheorem{proposition}[theorem]{Proposition}
\theoremstyle{definition}

\newtheorem{remark}[theorem]{Remark}

\numberwithin{equation}{section}
\newtheorem*{conjecture}{Conjecture}
\newtheorem*{problem}{Problem}
\newtheorem*{question}{Question}

\xyoption{all}

\def \a {\alpha }

\def \De {\Delta}
\def \la {\lambda}
\def \La {\Lambda}

\def\Om{\Omega}

\def\na {\nabla}

\begin{document}
	\baselineskip=17pt
	
	\titlerunning{Schrödinger Operators, Integral Curvature, and the Euler Characteristic of Riemannian Manifolds}
	\title{Schrödinger Operators, Integral Curvature, and the Euler Characteristic of Riemannian Manifolds}
	\author{Teng Huang and Pan Zhang}
	
	\date{}
	
	\maketitle
	
	\address{T. Huang: School of Mathematical Sciences, University of Science and Technology of China; CAS Key Laboratory of Wu Wen-Tsun Mathematics, University of Science  and Technology of China, Hefei, Anhui, 230026, P. R. China; \email{htmath@ustc.edu.cn;htustc@gmail.com}}
	\address{P. Zhang: School of Mathematical Sciences, Anhui University, Hefei, Anhui, 230026, People’s Republic of China; \email{panzhang20100@ahu.edu.cn}}

\begin{abstract}
We establish new connections between integral curvature bounds and the Euler characteristic of closed Riemannian manifolds through the perspective of Schr\"odinger-type operators.  Central to our approach is the twisted Dirac operator \(\mathcal{D}_{\theta}\), whose index equals \(\chi(M)\). Under integral smallness conditions on the negative part of a potential \(V\) and a Sobolev--Poincar\'e inequality, we show that a suitable scaling of \(\theta\) forces the kernel of \(\mathcal{D}_{t\theta}\) to vanish, thereby implying \(\chi(M)=0\).
	
Applying this framework to geometrically natural potentials yields several topological consequences. In even dimensions, sufficiently small integral bounds on partial sums of curvature operator eigenvalues force \(\chi(M)\) either to vanish or to have a sign determined by the middle dimension. For four-manifolds, a small \(L^{p}\)-norm of the negative Ricci curvature relative to the diameter guarantees \(\chi(M)\ge 0\). Moreover, when \(\chi(M)\neq 0\) we obtain a Li--Yau type lower bound for the first eigenvalue of the rough Laplacian on \(1\)-forms in terms of the diameter and an integral curvature quantity. Subsequently, we provide an explicit lower bound for the first eigenvalue of the Laplacian on $1$-forms under almost nonnegative curvature conditions, thereby giving an affirmative answer to Yau's Problem 79.
\end{abstract}

	\section{Introduction}
A central theme in Riemannian geometry is to understand how local geometric properties, particularly curvature, govern global topological structure. Classical results in global differential geometry---exemplified by the foundational works of Bochner \cite{Bochner1946}, Meyer \cite{Meyer1971}, Lichnerowicz \cite{Lichnerowicz1977}, and Gromov \cite{Gromov1978,Gromov1981}---typically rely on pointwise curvature conditions to derive topological or geometric consequences. For instance, Bochner's celebrated theorem states that if a closed Riemannian manifold satisfies Ricci curvature $\mathrm{Ric} \geq 0$ everywhere, then its first Betti number is bounded above by the dimension, with equality implying a flat metric structure \cite{Bochner1946}. Meyer's theorem  establishes that if $\mathrm{Ric} \geq n-1$, the diameter cannot exceed $\pi$, and Cheng's rigidity theorem shows that equality occurs precisely for the standard sphere \cite{Cheng1975}. Similarly, Lichnerowicz proved that under pointwise Ricci curvature lower bounds, the first nonzero eigenvalue of the Laplacian admits a sharp lower estimate. Meyer \cite{Meyer1971} established vanishing results for the Betti numbers of manifolds with positive curvature operators, i.e., $\la_{1}>0$, and in particular Meyer showed that they are rational (co)homology spheres.  More recently, Petersen and Wink \cite{Petersen2021} proved similar results using pointwise bound on the average of eigenvalues of the curvature operator. These results, while powerful and elegant, impose conditions that are often too restrictive for many geometrically interesting manifolds encountered in nature and applications. 

Recent work has shifted toward understanding the implications of integral curvature bounds. Gallot \cite{Gallot1981,Gallot1988} developed fundamental estimates for Sobolev constants under integral Ricci curvature bounds, establishing crucial analytic tools for this weaker geometric setting. Aubry established finiteness of $\pi_{1}$ and an upper bound of the diameter of a manifold in \cite{Aubry2007} using an integral bound of the lowest eigenvalue of the Ricci curvature pinched below a positive constant. Furthermore, when the diameter is also bounded from below, Aubry \cite{Aubry2009} proved that such manifolds are homeomorphic to $S^{n}$. Petersen, Sprouse, and Wei conducted a systematic investigation of geometric properties under bounded integral curvature, deriving volume comparison theorems, diameter estimates, and other comparison results \cite{Petersen1997a,Petersen1997b,Petersen1998}. Their work demonstrated that many classical comparison theorems admit natural integral analogs when appropriate $L^p$ conditions replace pointwise bounds. 

Although substantial progress has been made on bounding Betti numbers under integral curvature conditions (see \cite{Yu2022} for recent advances), the Euler characteristic has received less systematic attention in this context.  Some partial results exist: for instance, under pointwise conditions on the curvature operator, one can sometimes deduce sign constraints on $\chi(M)$ \cite{Bourguignon1978}. Nevertheless, a comprehensive theory linking integral bounds on the negative part of curvature to the Euler characteristic has remained elusive. The difficulty stems from several analytic obstacles: classical tools such as the Bochner formula and Moser iteration traditionally rely on pointwise curvature bounds, and extending them to the integral setting requires new technical innovations and careful control of constants expressed in terms of integral curvature quantities.

A powerful framework for analyzing curvature in integral norms arises from viewing certain geometric differential operators as Schrödinger-type operators 
$$\na^{\ast}\na+V,$$ 
where the curvature tensor appears as a potential $V$. This perspective transforms geometric problems into questions about the spectral and analytic properties of Schrödinger operators with $L^{p}$-potentials. In particular, the Hodge Laplacian on form
$$\De_{d}=\na^{\ast}\na+\operatorname{Ric}$$
is precisely of this form. 

In this article, we exploit this Schrödinger operator viewpoint to establish new links between integral curvature bounds and the Euler characteristic.  Our main technical vehicle is the twisted Dirac operator
\[
\mathcal{D}_{\theta}=d_{\theta}+d^{*}_{\theta}:\Omega^{+}(M)\to\Omega^{-}(M),
\]
where $d_{\theta}=d+\theta\wedge$ for a closed 1-form $\theta$. This operator, first studied systematically by Novikov \cite{Novikov1982} in the context of multi-valued functions and Hamiltonian systems, interpolates between the classical Hodge--Dirac operator (when $\theta=0$) and the differential in Morse--Novikov (twisted) cohomology theory. By the Atiyah--Singer index theorem, its index equals the Euler characteristic: $$\operatorname{Index}(\mathcal{D}_{\theta})=\chi(M).$$ 
Our strategy is to show that under suitable integral curvature conditions, one can select a harmonic 1-form $\theta$ (guaranteed by nonvanishing first de Rham cohomology) and a scaling parameter $t>0$ such that $\mathcal{D}_{t\theta}$ has trivial kernel in certain form degrees. Combining this analytic vanishing with Poincar\'{e} duality for Morse--Novikov cohomology \cite{Novikov1982} then yields strong constraints on $\chi(M)$. 
\begin{theorem}
	\label{thm:main1}
	Let $(M^{n},g)$ be a closed $n$-dimensional, $n\geq4$, Riemannian manifold satisfying the Sobolev--Poincar\'{e} inequality
	\[
	\|f-\bar{f}\|_{2}\leq C_{s}D\|df\|_{\frac{2p}{p+2}},\qquad p>2.
	\]
	Denote by $\lambda_{1}\leq\cdots\leq\lambda_{{n\choose 2}}$ the eigenvalues of the curvature operator and set 
	$$\lambda_{n,l}:=\lambda_{1}+\cdots+\lambda_{n-l}.$$ 
	Fix $D>0$ and $1\leq l\leq\left\lfloor\frac{n}{2}\right\rfloor$. 	Assume that $\theta$ is a non-trivial $1$-form satisfying 
	$$\nabla^{*}\nabla\theta+V\theta=0,$$
	where $V\in C^{\infty}\big(\operatorname{Sym}(T^{\ast}M)\big{)}$ is a potential. Then  for any $+\infty\geq q>p$, there exists a constant
	\begin{align*}
	\varepsilon=\min\bigg\{&e\sqrt{\|\la^{-}_{n,l}\|_{\frac{q}{2}}D^{2}}\exp\Bigl(-\sqrt{C_s}C_6(n,p,q)(\|\la^{-}_{n,l}\|_{\frac{q}{2}}D^{2})^{\frac{1}{4}}\Big),\\
	& \frac{1}{6\sqrt{e}C_s C_{1}(p,q)}\exp\Bigl(-\sqrt{C_s}C_6(n,p,q)(\|\la^{-}_{n,l}\|_{\frac{q}{2}}D^{2})^{\frac{1}{4}}\Bigr ),\\
	&\frac{1}{2C_s C_{4}(p,q)}\bigg\},\\
	\end{align*}
	such that if $$\sqrt{\|V^{-}\|_{\frac{q}{2}}D^{2}}\leq\varepsilon,$$
	one can find a positive real number \(t \in \mathbb{R}^{+}\) for which $$\ker\mathcal{D}_{t\theta}\cap\Omega^{\pm}_{l}(M)=0,$$
	where \[\Om^{+}_{l}=\bigoplus_{2k\leq l}(\Om^{2k}\oplus\Om^{n-2k}), \quad and\quad\Om^{-}_{l}=\bigoplus_{2k+1\leq l}(\Om^{2k+1}\oplus\Om^{n-2k-1})\] 
	Moreover, if $l=\lfloor\frac{n}{2}\rfloor$, then there exists a positive constant $t\in\mathbb{R}$ such that \[\ker\mathcal{D}_{t\theta}\cap\Omega^{\pm}(M)=0.\] In particular, the Euler characteristic of $M$ vanishes \[\chi(M)=0.\]
\end{theorem}
When the potential $V$ is related to geometric curvature (e.g., when $\theta$ is harmonic and $V$ involves the curvature operator), this theorem translates integral smallness of curvature into the vanishing of the Euler characteristic. 

Numerous important manifolds are almost positively curved, yet contain localized regions where curvature becomes negative, challenging the applicability of classical pointwise curvature theorems \cite{Gromov1981,Ruh1982,Yamaguchi1991}. The foundation for this direction was laid by Gromov’s groundbreaking work on almost flat manifolds \cite{Gromov1978} and his subsequent theorem bounding the total Betti number in terms of Ricci curvature lower bounds and diameter \cite{Gromov1981}. The investigation of manifolds with almost nonnegative curvature operator has been advanced in \cite{Herrmann2013}, and its applications to various topological invariants remain an active research area. In this context, the following question was raised:
\begin{question}(\cite[Question 4.6]{Herrmann2013})
Do manifolds with almost nonnegative curvature operator have nonnegative Euler characteristic?
\end{question}
Huang and Tan in \cite{Huang2025}, partially answered Question 4.6 proposed by Herrmann, Sebastian and Tuschmann by proving that when $\kappa D^{2}$ is sufficiently small, manifolds with nontrivial first de Rham cohomology have vanishing Euler characteristic, through their study of cancellation theorems in Morse--Novikov cohomology. 

Focusing on even-dimensional manifolds and employing finer curvature conditions, we obtain our second main result. This generalises the pointwise setting to an integral one via Theorem~\ref{thm:main1}.
\begin{theorem}
	\label{thm:main2}
	Let $(M^{2n},g)$ be a closed $2n$-dimensional, $n\geq2$, Riemannian manifold with nonzero first de Rham cohomology group, and let $q>p>2n$. Denote by $\lambda_{1}\leq\cdots\leq\lambda_{\binom{2n}{2}}$ the eigenvalues of the curvature operator and
	 set $$\lambda_{2n,l}:=\lambda_{1}+\cdots+\lambda_{2n-l}.$$ 
 Fix $D>0$ and $1\leq l\leq\left\lfloor\frac{n}{2}\right\rfloor$.  There exists a uniform constant $C_{10}(n,p,q)>0$ such that if 
	$$\|\lambda_{2n,l}^{-}\|_{\frac{q}{2}}D^{2}\leq C_{10},$$ 
	then the Euler characteristic of $M$ satisfies
	\begin{enumerate}
		\item $\chi(M)=0$, if $l=n$;
		\item $(-1)^{n}\chi(M)\geq 0$, if $l=n-1$.
	\end{enumerate}
\end{theorem}
Here $\lambda_{2n,l}^{-}$ denotes the negative part of sum of smallest $(2n-l)$ eigenvalues of the curvature operator. This theorem reveals a striking dichotomy in even dimensions: when the negative part of the curvature operator is sufficiently small in the integral sense, the Euler characteristic is forced either to vanish or to have a sign determined solely by the middle dimension. Moreover, we obtain an upper bound for the Betti numbers under a similar integral curvature condition.
\begin{theorem}\label{thm:harmonic-Linfty}
	Let \((M^{n},g)\) be a closed \(n\)-dimensional smooth Riemannian manifold, and fix \(D>0\) and an integer \(1\le l\le\bigl\lfloor\frac{n}{2}\bigr\rfloor\). 
	There exists a positive constant \(C_{9}(n,p)\) such that, if the curvature operator satisfies
	\[
	\bigl\|\lambda^{-}_{n,l}\bigr\|_{\frac{q}{2}}\,D^{2}\le C_{9}(n,p),\qquad q>p>n,
	\]
	then for every harmonic \(k\)-form \(\theta\in\mathcal{H}_{d}^{k}(M)\) with \(k\le l\) we have the \(L^{\infty}\)-estimate
	\[
	\|\theta\|_{\infty}
	\le \exp\Bigl\{ C_{4}\sqrt{\frac{n}{4}}\;C_{8}\sqrt{C_{9}} \Bigr\} \,\|\theta\|_{2}.
	\]
	where \(C_{4}=C_{4}(n,p,q)\) and \(C_{8}=C_{8}(n,p)\) are the constants appearing in Proposition~\ref{prop:schrodinger-Linfty} and Lemma~\ref{lem:sobolev-poincare-ricci}, respectively.
	
	In particular, one can choose a constant \(\varepsilon(n,p,q)>0\) such that, if \(C_{9}\le\varepsilon\), then the \(k\)-th Betti number satisfies
	\[
	b_{k}(M)\le\binom{n}{k}.
	\]
\end{theorem}

In dimension four, the Euler characteristic gains particular significance through its connections to fundamental topological and geometric structures. By the Hirzebruch signature theorem, $\chi(M)$ relates to the signature $\tau(M)$ via 
$$\chi(M)=\frac{1}{2}(\sigma(M)+\tau(M)),$$ 
where $\sigma(M)$ is a combination of Pontryagin numbers. Moreover, in the realm of smooth $4$-manifolds, Donaldson's gauge theory \cite{Donaldson1983,Donaldson1987} and Seiberg--Witten theory \cite{Witten1994} reveal deep connections between $\chi(M)$, the intersection form, and differentiable structures. These connections make understanding which curvature conditions force $\chi(M)$ to have a specific sign---especially nonnegativity---a problem of considerable geometric and topological interest. Here we obtain a clean and geometrically transparent statement involving only the Ricci curvature:
\begin{corollary}
	\label{cor:4dim}
	Let $(M^{4},g)$ be a closed $4$-dimensional smooth Riemannian manifold, and let $q>p>4$. There exists a uniform positive constant $C_{10}(p)$ such that if the Ricci curvature satisfies 
	$$\|{\rm Ric}^{-}\|_{\frac{q}{2}}D^{2}\leq C_{10},$$
	then the Euler characteristic of $M$ is non‑negative:
	\[\chi(M)\geq 0.\]
\end{corollary}
This provides a quantitative obstruction: any $4$-manifold with $\chi(M)<0$ cannot admit a metric whose negative Ricci curvature is too small in an integral sense relative to the diameter. 
\begin{remark}
The significance of the above result can be clarified by the example constructed by Anderson \cite{Anderson1992}. He showed that for integers $n\geq 4$, $k\leq n-1$ and a constant $\varepsilon>0$, there exist closed $n$-manifolds $M$ satisfying 
$$|{\rm Ric}(g)|<\varepsilon,\quad{\rm diam}(M)<1,$$ and 
$b_{1}(M)=k$, yet no finite cover of $M$ fibers over the circle $S^{1}$. Moreover, these manifolds can be chosen with arbitrarily large second Betti number. For $n=4$, the Euler characteristic of these manifolds satisfies $$\chi(M)=2-2b_{1}(M)+b_{2}(M)>0.$$ 
Our corollary thus establishes that such constructions are essentially optimal: if one wants to force $\chi(M)<0$, a certain minimal amount of negative Ricci curvature in the $L^p$ sense is necessary.
\end{remark}
We consider the \textit{rough Laplacian} 
$$\overline{\De}=\na^{\ast}\na$$
acting on differential $k$-forms.  This is a second-order elliptic operator whose spectrum consists of a nonnegative, unbounded, 
increasing sequence of real eigenvalues: 
$$0=\la^{(k)}_{0}<\la^{(k)}_{1}\leq\la^{(k)}_{2}\leq\cdots\leq\la^{(k)}_{l}\leq\la^{(k)}_{l+1}\leq\cdots\nearrow+\infty,$$
where the geometric multiplicity of $\la^{(k)}_{0}$ equals the dimension of $\dim\ker\na$ (i.e., the dimension of parallel $k$-forms). 
If there are no nonzero parallel forms, the spectrum conventionally begins with the first positive eigenvalue $\la^{(k)}_{1}>0$.

For the first positive eigenvalue $\la_{1}^{(0)}$ of functions (i.e., $0$-forms), a well-established geometric estimate theory exists. 
The work of Cheeger \cite{Cheeger1970} and Buser \cite{Buser1982} shows that 
$\la_{1}^{(0)}$  is closely related to the Cheeger isoperimetric constant of the manifold; 
whereas the gradient estimate technique developed by Li and Yau \cite{Li1980b} expresses its lower bound explicitly as a function of a lower bound on Ricci curvature and the diameter. 
Specifically, if the Ricci curvature satisfies
$$\operatorname{Ric}(g)\geq -(n-1)\kappa,\quad and\quad{\rm{diam}}(M)\leq D(g),$$ 
then
	\begin{equation*}
	\la^{(0)}_{1}D^{2}(g)\geq c^{-1}(n)\exp\{-[1+(1+2c^{2}(n)\La^{2})^{\frac{1}{2} }]\}.
	\end{equation*}	
The authors in \cite{Zhong1984} improved the result of Li-Yau. However, the situation for the eigenvalues $\la_{1}^{(k)}$ of rohigher-degree differential forms $1\leq k\leq n-1$  is quite different.
Colbois and Maerten \cite{Colbois2010} pointed out that fixing the volume alone does not guarantee a uniform positive lower bound for $\la_{1}^{(k)}$; 
moreover, the constructions of Ann\'{e} and Takahashi \cite{Anne2024,Anne2025} demonstrate that on any closed manifold, 
there exists a family of metrics with fixed volume such that the $l$-th positive eigenvalue of the rough Laplacian (or the Hodge Laplacian) acting on $k$-forms can be made arbitrarily close to zero.

These results clearly indicate that in order to obtain a positive lower bound for the eigenvalues of higher-degree forms, additional geometric or topological constraints must be imposed. Existing work has introduced such constraints from various perspectives: for instance, Ballmann, Br\"{u}ning and Carron \cite{Ballmann2003} considered holonomy restrictions; Lipnowski and Stern \cite{Lipnowski2018} studied geometric estimates for the Laplacian on $1$-forms on closed hyperbolic manifolds; and Boulanger and Courtois \cite{Boulanger2022} established a Cheeger-type inequality for coexact $1$-forms on closed orientable Riemannian manifolds.

Finally, we provide a lower bound estimate for the first eigenvalue $\lambda_{1}^{(1)}$ of the rough Laplacian on $1$-forms. 
As noted in \cite[Proposition 1.2]{Huang2025b}, a Li-Yau type estimate for $\lambda_{1}^{(1)}$ does not hold in general under only a pointwise Ricci lower bound.  

By strengthening the curvature assumption to a pointwise lower bound on a partial sum of the eigenvalues of the curvature operator, and assuming the manifold has nonzero Euler characteristic,  we are able to obtain a Li-Yau-type inequality for the eigenvalue $\la^{1}_{(1)}$.
\begin{theorem}
	\label{thm:eigenvalue}
	Let $(M^{2n},g)$ be a closed $2n$-dimensional, $n\ge2$, smooth Riemannian manifold with nonvanishing Euler characteristic. Denote by $\lambda_{1}\leq\cdots\leq\lambda_{\binom{2n}{2}}$ the eigenvalues of the curvature operator and
	$$\la_{2n,n}=\lambda_{1}+\cdots+\lambda_{n}.$$
	 Then the first eigenvalue of the rough Laplacian on $1$-forms satisfies	
\begin{equation} \label{V8}
\begin{split}
	\sqrt{\lambda_{1}^{(1)}D^{2}}
	>\min\bigg\{&e\sqrt{\|\la^{-}_{2n,n}\|_{\infty}D^{2}}\exp\Bigl(-\sqrt{C_s(\infty)}C_6(n,2n,\infty)(\|\la^{-}_{2n,n}\|_{\infty}D^{2})^{\frac{1}{4}}\Big),\\
	& \frac{1}{6\sqrt{e}C_s(\infty) C_{1}(n,\infty)}\exp\Bigl(-\sqrt{C_s(\infty)}C_6(n,2n,\infty)(\|\la^{-}_{2n,n}\|_{\infty}D^{2})^{\frac{1}{4}}\Bigr ),\\
	&\frac{1}{2C_s(\infty) C_{4}(n,\infty)}\bigg\},
	\end{split}
	\end{equation}
 where \[C_s(\infty) = C(n)\exp\Bigl(\frac{2n-1}{\sqrt{n}}
 \sqrt{\|\lambda^{-}_{2n,n}\|_{\infty}D^{2}}\Bigr).\]
\end{theorem}
The eigenvalue estimate above leads to an alternative conclusion concerning the Euler characteristic and the first Betti number when the quantity \(\|\lambda_{2n,n}^{-}\|_{\infty}D^{2}\) is sufficiently small.
\begin{corollary}\label{cor:euler-betti-choice}
	Let $(M^{2n},g)$ be a closed $2n$-dimensional, $n\geq2$, smooth Riemannian manifold. Denote by $\lambda_{1}\leq\cdots\leq\lambda_{\binom{2n}{2}}$ the eigenvalues of the curvature operator and set 
	\[
	\lambda_{2n,n}:=\lambda_{1}+\cdots+\lambda_{n}.
	\]
	There exists a uniform constant $C_{11}(n)>0$ such that if 
	\[
	\|\lambda_{2n,n}^{-}\|_{\infty}D^{2}\leq C_{11}(n),
	\]
	then either $\chi(M)=0$ or $b_{1}(M)=0$.
\end{corollary}
In his list of famous problems \cite{Yau1982} (or see \cite{Schoen1994}), S.-T. Yau posed the following problem:
\begin{problem}(\cite[Problem 79]{Yau1982})
Let $M$ be a smooth closed manifold. How can one estimate the first nonzero eigenvalue of the Hodge-Laplacian on differential forms in terms of computable geometric quantities?
\end{problem}
The spectrum of $\Delta_d$ acting on $p$-forms will be denoted by
\[
0 = \mu^{(p)}_0 < \mu^{(p)}_1 \le \mu^{(p)}_2 \le \cdots \le \mu^{(p)}_k \le \mu^{(p)}_{k+1} \le \cdots \nearrow +\infty .
\]
Colbois and Courtois proved in \cite[Theorem 0.4]{Colbois1990} that if a closed $n$-dimensional Riemannian manifold $(M,g)$ satisfies
\[
|{\rm sec}_g| \le \kappa, \qquad {\rm diam}(M) \le D, \qquad {\rm Vol}(g) \ge \nu,
\]
then there exists a constant $C(n,\kappa,\nu,D) > 0$ such that for every $p=0,1, \dots, n-1$,
\[
\mu^{(p)}_1 \ge C(n,\kappa,\nu,D).
\]
The main idea of their proof relies on the compactness of the class of manifolds satisfying the assumptions of the theorem.  If we weaken the curvature assumption of $|{\rm sec}_g| \le \kappa$ by ${\rm sec}_g\ge \kappa$, then we do not know
whether or not such a positive lower bound exists. But, Lott  conjectured the
following:
\begin{conjecture} (\cite[Lott]{Lott2004})
For given $n\in\mathbb{N}$, $\kappa\in\mathbb{R}$ and $\nu, D>0$, there would exist
a positive constant $C(n,\kappa,\nu,D)>0$ depending only on $n,\kappa,\nu$ and $D$ such that any
connected oriented closed Riemannian manifold $(M^{n},g)$ of dimension $n$ with 
\[
{\rm sec}_g \ge \kappa, \qquad {\rm diam}(M) \le D, \qquad {\rm Vol}(g) \ge \nu,
\]
satisfies
\[
\mu^{(p)}_1 \ge C(n,\kappa,\nu,D),\qquad \forall n=0,1,\cdots,n.
\]
\end{conjecture}
This conjecture is still open, as far as the authors know. Recently, Honda and Mondino \cite{Honda2025} obtained a positive lower bound of $\mu^{(1)}_{1}$ for $p= 1$ under $n\le4$, the Ricci
curvature $|{\rm Ric}(g)| \le \kappa, {\rm diam}(M) \le D, {\rm Vol}(g) \ge \nu.$

In \cite{Boulanger2022}, the authors obtain a lower bound for the smallest eigenvalue of the Laplacian on coexact $1$-forms in terms of an isoperimetric constant related to one-dimensional cycles. 

When the eigenvalues of the curvature operator are sufficiently small in the $L^{\infty}$-norm and the Euler characteristic is non‑zero, we can establish a lower bound for the first non‑zero eigenvalue of the Hodge Laplacian acting on  $1$‑forms.
\begin{theorem}\label{thm:eigenvalue-Hodge-Laplacian}
	Let $(M^{2n},g)$ be a closed $2n$-dimensional, $n\ge 2$, smooth Riemannian manifold with nonvanishing Euler characteristic. Denote by $\lambda_1\le\cdots\le\lambda_{\binom{2n}{2}}$ the eigenvalues of the curvature operator and set
	\[
	\lambda_{2n,n} :=\lambda_1+\cdots+\lambda_n .
	\]
	There exists a uniform constant $C_{11}(n)>0$ such that if
	\[
	\|\lambda_{2n,n}^{-}\|_{\infty}\, D^{2}\le C_{11}(n),
	\]
	then the first eigenvalue of the Hodge–Laplacian on $1$-forms satisfies
	\[
	\mu_{1}^{(1)}
	> \|\lambda^{-}_{2n,n}\|_{\infty}(e-2).
	\]
\end{theorem}
\begin{remark}
A smooth closed  manifold $M$ is said to admit almost nonnegative curvature operator (ANCO) if there exists a sequence of Riemannian metrics $\{g_{j}\}$ on $X$ such that all eigenvalues
	$\la^{(j)}_{i}=\la_{i}(g_{j})$ of the associated curvature operator and the diameter ${\rm{diam}}(g_{j})$ of $g_{j}$ satisfy
	$$\la^{(j)}_{i}\geq -\frac{1}{j},\ {\rm{diam}}(g_{j})\leq 1.$$
It is then easy to verify that ANCO manifolds satisfy the curvature conditions stated in Theorem \ref{thm:eigenvalue-Hodge-Laplacian}. Moreover,  in the case of even-dimensional ANCO manifolds, we give an answer to Yau’s 79th problem concerning 1‑forms.
\end{remark}

\subsection*{Outline of the article}
Section 2 presents preliminary material on curvature operators and establishes a key analytic result---a mean-value inequality under integral curvature conditions. By deriving a Bochner-type formula for forms in \(\ker\mathcal{D}_{\theta}\) and coupling it with a delicate Moser iteration (Proposition \ref{prop:mean-value}) that operates under only integral curvature bounds, we extend the pointwise framework of classical Bochner theory to the integral setting.  

Section 3 introduces the twisted Dirac operator \(\mathcal{D}_{\theta}\), proves fundamental integral identities, and develops uniform \(L^{\infty}\)-estimates for forms in its kernel. Working with \(\mathcal{D}_{\theta}\) allows us to link analytic information (the kernel of \(\mathcal{D}_{t\theta}\)) directly to topology via the index theorem. The use of Morse--Novikov cohomology and its Poincar\'e duality \cite{Novikov1982} plays an essential role in deducing sign constraints on \(\chi(M)\).  

Section 4 provides the proofs of the main theorems and their corollaries, including applications to eigenvalue estimates. In our analysis we rely on and extend foundational work of Gallot \cite{Gallot1988} concerning Sobolev constants under integral Ricci bounds (Lemma \ref{lem:sobolev-poincare-ricci}), which requires careful control of constants expressed in terms of integral curvature quantities. 

Section 5 establishes eigenvalue estimates for the rough and Hodge Laplacians under almost nonnegative curvature conditions. In particular, we prove a Li–Yau type lower bound for the first eigenvalue of the rough Laplacian on 1-forms (Theorem \ref{thm:eigenvalue}), and show that when the negative part of the curvature operator is sufficiently small in the $L^{\infty}$-sense, either the Euler characteristic vanishes or the first Betti number vanishes (Corollary \ref{cor:euler-betti-choice}). Moreover, we obtain an explicit positive lower bound for the first eigenvalue of the Hodge Laplacian on 1-forms (Theorem \ref{thm:eigenvalue-Hodge-Laplacian}), answering a question of Yau under our curvature setting. 

\subsection*{Notational Conventions}
Throughout, \((M^{n}, g)\) denotes a closed oriented smooth \(n\)-dimensional Riemannian manifold. For a function \(f: M \to \mathbb{R}\), we define its positive and negative parts by
\[
f^{+}(x) = \max\{0, f(x)\}, \quad f^{-}(x) = \max\{0, -f(x)\}.
\]
The normalized \(L^p\)-norm of \(f\) is defined as
\[
\|f\|_p = \left(\frac{1}{\operatorname{Vol}(g)} \int_M |f|^p \, d\mathrm{vol}_g\right)^{\frac{1}{p}},
\]
where \(\operatorname{Vol}(g)\) denotes the volume of \((M,g)\). The diameter of \((M, g)\) is denoted by \(D\).

Let \(\lambda_1 \le \dots \le \lambda_{\binom{n}{2}}\) be the eigenvalues of the curvature operator of \((M,g)\). We denote
\[
\lambda_{n,l} := \lambda_1 + \dots + \lambda_{n-l}
\]
as the sum of the smallest \(n-l\) eigenvalues, where \(1 \leq l \leq \lfloor\frac{n}{2}\rfloor\).

In what follows, let \(+\infty \geq q > p > 2\). We introduce the constants
\[
\gamma := \frac{p(q-2)}{(p-2)q} > 1,
\]
and define
\[C_{0}(p,q):=\frac{q-2}{2q}\Big(\frac{2p}{p-1}+\frac{pq}{q-p}\Bigr).\]
\[
C_{1}(p,q) := \exp\left\{\frac{4p}{(p-2)(\gamma-1)\sqrt{\gamma}}\right\},
\]
\[C_4(p,q):=\frac{2p}{p-2}\sqrt{C_{0}(p,q)} 
\frac{1}{\sqrt{\gamma}(\sqrt{\gamma}-1)} ,\]
\[
C_5(n,p,q) := \bigl(9e^5 C_2^2(n) C^2_1(p,q)+1\bigr),
\]
\[
C_{6}(n,p,q):=\frac{4pq}{q-p}\sqrt{C_{0}(p,q)}(C_{5}(p,q))^{\frac{1}{4}}.
\]
Furthermore, set
 \[C_{10}(n,p,q):=
\min\!\Bigl\{\frac{1}{16C^{2}_{8}(n,p)C^{4}_{6}(n,p,q)},\frac{1}{72e^{2}C^{2}_8(n,p) C^{2}_{1}(p,q)},\frac{1}{8C_8^{2}(p,q) C^{2}_{4}(p,q)},\frac{1}{2}C_{7}(n,p)\Bigr\},
\]
where \(C_7(n,p)\) and \(C_8(n,p)\) are positive constants depending only on $n$ and $q$. When $q<+\infty$, we can choose $p=\frac{q+n}{2}$.

In the limiting case \(q = +\infty\), we have
\[
\gamma = \frac{p}{p-2}, \quad \lim_{q\rightarrow+\infty}\frac{q-p}{2pq} = \frac{1}{2p}, \]
\[C_{0}=\frac{p}{p-1}+\frac{p}{2},\quad C_1 = \exp\bigl\{2p\sqrt{p(p-2)}\bigr\},
\]
\[
C_4 =\sqrt{p}(\sqrt{p}+\sqrt{p-2})\sqrt{\frac{p}{p-1}+\frac{p}{2}}.
\]

	\section{Preliminaries}
	
	\subsection{Curvature Operator} 
	We first recall the definition of the curvature operator and recent progress in this area due to Petersen and Wink \cite{Petersen2021}.
	
	Let \((M,g)\) be a closed \(n\)-dimensional Riemannian manifold and let 
	\[
	R(X,Y)Z = \nabla_Y\nabla_X Z - \nabla_X\nabla_Y Z + \nabla_{[X,Y]}Z
	\] 
	denote its curvature tensor. The Weitzenb\"ock curvature operator acting on a \((0,k)\)-tensor \(T\) is defined by
	\[
	\operatorname{Ric}(T)(X_1,\dots,X_k) = \sum_{i=1}^k \sum_{j=1}^n \bigl(R(X_i,e_j)T\bigr)(X_1,\dots,e_j,\dots,X_k).
	\]
	We use a variation of the Ricci tensor to symbolize this operator, as it indeed coincides with the classical Ricci tensor when evaluated on vector fields or 1-forms \cite{Petersen1998}.
	
Let \((E,\nabla)\) be a Riemannian vector bundle with finite-dimensional fibre over \(M\); that is, \(E\) is a vector bundle equipped with a smooth metric \(\langle\cdot,\cdot\rangle\) and a compatible connection \(\nabla\). On the space \(\Gamma(E)\) of smooth sections of \(E\) we define the \(L^{2}\)-inner product \((\cdot,\cdot)\) induced by \(\langle\cdot,\cdot\rangle\) and the metric \(g\). The connection \(\nabla\) extends naturally to \(p\)-tensors on \(M\) taking values in \(E\). Its formal adjoint \(\nabla^{*}\) with respect to the \(L^{2}\)-inner product defines the \emph{rough Laplacian} (or connection Laplacian) acting on \(\Gamma(E)\) by
\[
\overline{\Delta} = \nabla^{*}\nabla .
\]
Suppose now that \(E \to M\) is a subbundle of \(T^{(0,k)}(M)\). For a constant \(c>0\) the \emph{Lichnerowicz Laplacian} on \(E\) is given by
\[
\Delta_{L} = \nabla^{*}\nabla + c\operatorname{Ric}.
\]
The Hodge Laplacian corresponds to the case \(c = 1\); that is,
\[
\Delta_{d} =\overline{\Delta}+ \operatorname{Ric}.
\]

The following estimate for the curvature term plays a crucial role in our analysis.

	\begin{corollary}[{\cite[Lemma 2.1]{Petersen2021}}]\label{cor:eigenvalue-estimate}
		Let \((M^{n},g)\) be a closed \(n\)-dimensional Riemannian manifold and let  
		\(\lambda_{1}\le\cdots\le\lambda_{\binom{n}{2}}\) be the eigenvalues of its curvature operator.  
		Fix an integer \(1\le l\le\big\lfloor\frac{n}{2}\big\rfloor\).  
		Then for any differential form \(\alpha\in\Omega^{k}(M)\) with \(k\le l\) or \(k\ge n-l\),
		\begin{equation}\label{eq:curvature-bound}
		\langle\operatorname{Ric}(\alpha),\alpha\rangle
		\ge \frac{\lambda_{1}+\dots+\lambda_{n-l}}{n-l}\;k(n-k)\,|\alpha|^{2}.
		\end{equation}
	\end{corollary}
	
	\begin{remark}\label{Rem}
		Because the Ricci curvature is bounded from below by the sum of the smallest \(n-1\) eigenvalues of the curvature operator, the estimate \eqref{eq:curvature-bound} yields the pointwise inequality
		\begin{equation}\label{rem}
		\operatorname{Ric}(g) \ge (n-1)\,\frac{\lambda_{1}+\cdots+\lambda_{n-l}}{n-l}. 
		\end{equation}
	\end{remark}
	
	\subsection{A Mean Value Inequality}
	The following elementary estimate will be used repeatedly in the Moser iteration.
	
	\begin{lemma}\label{lem:product-estimate}
		Let \(t > 0\), \(\zeta > 1\). Then
		\[
		P = \prod_{i=0}^{\infty} \bigl(1 + t\zeta^{i}\bigr)^{\frac{1}{\zeta^{i}}}
		\le \exp\Bigl\{\frac{2\sqrt{\zeta}}{\zeta-1}\Bigr\}
		\bigl(1+\sqrt{t}\bigr)^{\frac{2\zeta}{\zeta-1}} .
		\]
	\end{lemma}
	
	\begin{proof}
		For any \(x > 0\), we have
		\begin{align*}
		\ln(1+tx) &\le 2\ln\bigl(1+\sqrt{tx}\bigr) \\
		&= 2\ln\bigl(1+\sqrt{t}\bigr) 
		+ 2\ln\Bigl(1+(\sqrt{x}-1)\frac{\sqrt{t}}{1+\sqrt{t}}\Bigr) \\
		&\le 2\ln\bigl(1+\sqrt{t}\bigr) + 2(\sqrt{x}-1).
		\end{align*}
		Therefore,
		\begin{align*}
		\ln P &= \sum_{i=0}^{\infty} \frac{\ln\bigl(1+t\zeta^{i}\bigr)}{\zeta^{i}} \\
		&\le \sum_{i=0}^{\infty} \frac{2\ln\bigl(1+\sqrt{t}\bigr)}{\zeta^{i}}
		+ \sum_{i=0}^{\infty} \frac{2\zeta^{\frac{i}{2}}-2}{\zeta^{i}} \\
		&\le \frac{2\zeta\ln\bigl(1+\sqrt{t}\bigr)}{\zeta-1}
		+ \frac{2\sqrt{\zeta}}{\sqrt{\zeta}-1} - \frac{2\zeta}{\zeta-1} \\
		&= \frac{2\zeta\ln\bigl(1+\sqrt{t}\bigr)}{\zeta-1} + \frac{2\sqrt{\zeta}}{\zeta-1}.
		\end{align*}
		Exponentiating both sides gives the desired inequality.
	\end{proof}
	
	The next proposition is the analytic core of our \(L^\infty\)-estimates. It provides a uniform bound for nonnegative functions satisfying a differential inequality with coefficients in \(L^p\). We now present the proof of Proposition \ref{prop:mean-value},  whose main technical idea originates from \cite{Aubry2009,Aubry2003}.
	
	\begin{proposition}[Mean Value Inequality]\label{prop:mean-value}
		Let \((M^{n},g)\) be a closed \(n\)-dimensional smooth Riemannian manifold satisfying the Sobolev inequality
		\begin{equation}\label{eq:sob-ineq}
		\|f\|_{\frac{2p}{p-2}} \le \|f\|_2 + C_s D \|df\|_2,\qquad p>2.
		\end{equation}
		Suppose that \(f \in L^2_1(M)\) is a nonnegative continuous function satisfying
		\begin{equation}\label{eq:diff-ineq}
		f\Delta_d f \le h_1 f^2 + \operatorname{div} X
		\end{equation}
		in the sense of distributions, where \(h_1 \in L^\frac{q}{2}(M)\) and \(X\) is a vector field with
		\begin{equation}\label{eq:V-bound}
		|X|(x) \le h_2(x) f^2(x), \qquad \forall x \in M,
		\end{equation}
		for a nonnegative continuous function \(h_2 \in L^{q}(M)\). Then for any \(q\) with {\(+\infty\geq q > p\)},
		\begin{equation}\label{eq:Linfty-bound}
		\|f\|_\infty \le C_1(p,q)\,
		\bigl(1 + \sqrt{C_{0}(p,q)C_s D B}\bigr)^{\frac{2pq}{q-p}}\,
		\|f\|_2,
		\end{equation}
		where 
		\[
		B^2 := \|h_1\|_\frac{q}{2} + 2\|h_2\|_{q}^2,\quad and\quad C_{1}(p,q)= \exp\left\{\frac{4p}{(p-2)(\gamma-1)\sqrt{\gamma}}\right\}
		\]
	\end{proposition}
	
	\begin{proof}
		For any integer \(k \ge 1\), multiply inequality \eqref{eq:diff-ineq} by \(f^{2k-2}\) and integrate over \(M\). Using the divergence theorem, we obtain
		\[
		\int_M f^{2k-1}\Delta_d f \le 
		\|h_1\|_{L^\frac{q}{2} } \|f\|_{L^{\frac{2kq}{q-2}}}^{2k}
		+ \bigl|(\operatorname{div} X, f^{2k-2})_{L^2(M)}\bigr|.
		\]
		For the divergence term we estimate
		\begin{align*}
		\bigl|(\operatorname{div} X, f^{2k-2})_{L^2(M)}\bigr|
		&= \bigl|(X, \nabla f^{2k-2})_{L^2(M)}\bigr| \\
		&= (2k-2)\bigl|(X, f^{2k-3}\nabla f)_{L^2(M)}\bigr| \\
		&\le (2k-2)\int_M f^{2k-1}|h_2|\,|\nabla f| \\
		&\le \frac{2k-2}{4}\int_M f^{2k-2}|\nabla f|^2 
		+ (2k-2)\int_M f^{2k}|h_2|^2 .
		\end{align*}
		Combining the two estimates gives
		\begin{align*}
		(2k-1)\int_{M}f^{2k-2}|\na f|^{2}&=\int_{M}f^{2k-1}\De_{d}f\\
		&\leq  \|h_{1}\|_{L^{\frac{q}{2}}}\|f\|^{2k}_{L^{\frac{2kq}{q-2}}}\\
		&\quad+\frac{(2k-2)}{4}\int_{M}f^{2k-2}|\na f|^{2}\\
		&\quad+(2k-2)\|h_{2}\|^{2}_{L^{q}}\|f\|^{2k}_{L^{\frac{2kq}{q-2}}}.\\
		\end{align*}
		Rearranging yields
		\begin{equation}\label{V6}
		\int_M f^{2k-2}|\nabla f|^2 \le 
		\Bigl(\frac{4(k-1)}{3k-1}\|h_2\|_{L^{q}}^2 +\frac{2}{3k-1} \|h_1\|_{L^\frac{q}{2}}\Bigr)
		\|f\|_{L^{\frac{2kq}{q-2}}}^{2k}.
		\end{equation}   
		Consequently,
		\begin{align*}
		\int_M |d f^k|^2 &= k^2\int_M f^{2k-2}|\nabla f|^2 \\
		&\le k^2 B^2 \|f\|_{L^{\frac{2kq}{q-2}}}^{2k},
		\end{align*}
		i.e.,
		\begin{equation}\label{eq:gradient-bound}
		\|d f^k\|_2 \le k B \|f\|_{\frac{2kq}{q-2}}^k .
		\end{equation}
		{
		Now apply the Sobolev inequality \eqref{eq:sob-ineq} to \(f^k\):
		\begin{align*}
		\|f\|_{\frac{2pk}{p-2}}^k 
		&= \|f^k\|_{\frac{2p}{p-2}} \\
		&\le \|f^k\|_2 + C_s D \|d(f^k)\|_2 \\
		&\le \|f^k\|_2 + C_s D k B \|f\|_{\frac{2kq}{q-2}}^k .
		\end{align*}
		Using the interpolation inequality
		\[
	\|f\|_{2k}\le	\|f\|_{\frac{2kq}{q-2}} \le \|f\|_\infty^{\frac{1}{k}} \|f\|_{\frac{2(k-1)q}{q-2}}^{1-\frac{1}{k}},
		\]
		we obtain
		\begin{align*}
		\|f\|_{\frac{2pk}{p-2}}^k 
		&\le \Bigl(\|f\|_\infty^{\frac{1}{k}} \|f\|_{\frac{2(k-1)q}{q-2}}^{1-\frac{1}{k}}\Bigr)^k
		+ C_s D k B \Bigl(\|f\|_\infty^{\frac{1}{k}} \|f\|_{\frac{2(k-1)q}{q-2}}^{1-\frac{1}{k}}\Bigr)^k \\
		&= \bigl(1 + C_s D k B\bigr)
		\|f\|_\infty \|f\|_{\frac{2(k-1)q}{q-2}}^{k-1}.
		\end{align*}    
		Hence,
		\[
		\frac{\|f\|_{\frac{2pk}{p-2}}}{\|f\|_\infty}
		\le \bigl(1 + C_s D k B\bigr)^{\frac{1}{k}}
		\Bigl(\frac{\|f\|_{\frac{2(k-1)q}{q-2}}}{\|f\|_\infty}\Bigr)^{\frac{k-1}{k}} .
		\]    
		Raising both sides to the power \(\frac{2pk}{p-2}\) gives the iterative estimate
		\begin{equation}\label{eq:iteration}
		\Bigl(\frac{\|f\|_{\frac{2pk}{p-2}}}{\|f\|_\infty}\Bigr)^{\frac{2pk}{p-2}}
		\le \bigl(1 + C_s D k B\bigr)^{\frac{2p}{p-2}}
		\Bigl(\frac{\|f\|_{\frac{2(k-1)q}{q-2}}}{\|f\|_\infty}\Bigr)^{\frac{2p(k-1)}{p-2}} .
		\end{equation}    
		Now define the sequences
		\[
		a_0 = \frac{2p}{p-1}, \qquad
		a_{i+1} = \gamma a_i + \frac{2p}{p-2},
		\]
		where 
		\begin{equation*}
			\gamma = \frac{p(q-2)}{(p-2)q} > 1.
		\end{equation*}
		One easily checks that
		\[
		a_i = \gamma^i\bigl(a_0 + \gamma_0\bigr) - \gamma_0, \qquad
		\gamma_0 = \frac{2p}{(p-2)(\gamma-1)} =\frac{pq}{q-p}.
		\]
		Set
		\[
		k_i = \frac{q-2}{2q}a_i + 1,
		\]
		which satisfies
		\[
		\frac{2p}{p-2}k_i = \gamma a_i + \frac{2p}{p-2} = a_{i+1},\quad and\quad \frac{2p(k_{i}-1)}{p-2}=\gamma a_{i}.
		\]
		Applying \eqref{eq:iteration} with \(k = k_i\) yields
		\[
		\Bigl(\frac{\|f\|_{a_{i+1}}}{\|f\|_\infty}\Bigr)^{\frac{a_{i+1}}{\gamma^{i+1}}}
		\le \bigl(1 + C_s D k_i B\bigr)^{\frac{2p}{(p-2)\gamma^{i+1}}}
		\Bigl(\frac{\|f\|_{a_i}}{\|f\|_\infty}\Bigr)^{\frac{a_i}{\gamma^i}} .
		\]
		Iterating this inequality backward to \(i=0\) gives
		\[
		\Bigl(\frac{\|f\|_{a_{i+1}}}{\|f\|_\infty}\Bigr)^{\frac{a_{i+1}}{\gamma^{i+1}}}
		\le \prod_{\ell=0}^{i} \bigl(1 + C_s D k_\ell B\bigr)^{\frac{2p}{(p-2)\gamma^{\ell+1}}}
		\Bigl(\frac{\|f\|_{a_0}}{\\|f\|_\infty}\Bigr)^{a_0}.
		\]    
		Taking the limit \(i\to\infty\) and noting that \(\|f\|_{a_i} \to \|f\|_\infty\), we obtain
		\begin{align*}
		1 &\le \prod_{\ell=0}^{\infty} \bigl(1 + C_s D k_\ell B\bigr)^{\frac{2p}{(p-2)\gamma^{\ell+1}}}
		\Bigl(\frac{\|f\|_{a_0}}{\|f\|_\infty}\Bigr)^{a_0} \\
		&\le \prod_{\ell=0}^{\infty} \bigl(1 + C_0(p,q) D \gamma^{\ell} B\bigr)^{\frac{2p}{(p-2)\gamma^{\ell+1}}}
		\Bigl(\frac{\|f\|_{a_0}}{\|f\|_\infty}\Bigr)^{a_0}.
		\end{align*}    
		Here we use the fact
		\[k_{i}<\gamma^{i}(a_{0}+\gamma_{0}):=\gamma^{i}C_{0}(p,q).\]
		Using the interpolation 
		$$\|f\|_{a_0} \le \|f\|_2^{1-\frac{1}{p}} \|f\|_\infty^{\frac{1}{p}},$$ we have
		\[
		\Bigl(\frac{\|f\|_{a_0}}{\|f\|_\infty}\Bigr)^{a_0}
		\le \Bigl(\frac{\|f\|_2^{1-\frac{1}{p}} \|f\|_\infty^{\frac{1}{p}}}{\|f\|_\infty}\Bigr)^{a_0}
		= \Bigl(\frac{\|f\|_2}{\|f\|_\infty}\Bigr)^{2}.
		\]
		Therefore,
		\[
		\frac{\|f\|_\infty^2}{\|f\|_2^2}
		\le \prod_{\ell=0}^{\infty} 
		\bigl(1 + C_{0}(p,q)C_s D \gamma^{\ell} B\bigr)^{\frac{2p}{(p-2)\gamma^{\ell+1}}}.
		\]
	}Applying Lemma \ref{lem:product-estimate} with \(t = C_0(p,q)C_s D B\) gives
		\begin{equation*}
		\begin{split}
		\frac{\|f\|_\infty^2}{\|f\|_2^2}
		&\le\Bigl [\exp\Bigl\{\frac{2\sqrt{\gamma}}{\gamma-1}\Bigr\}
		\bigl(1+\sqrt{t}\bigr)^{\frac{2\gamma}{\gamma-1}}\Bigr]^{\frac{2p}{(p-2)\gamma}}\\
		&=\exp\Bigl\{\frac{4p}{(p-2)(\gamma-1)\sqrt{\gamma}}\Bigr\}
		\bigl(1 + \sqrt{C_{0}(p,q)C_s D B}\bigr)^{\frac{4p}{(p-2)(\gamma-1)}}.		
		\end{split}
		\end{equation*}
		Substituting the values of \(\gamma\) and \(\gamma-1\) yields the desired bound \eqref{eq:Linfty-bound} with a constant \(C_1(p,q)\) depending only on \(p\) and \(q\).
	\end{proof}

	\section{The Twisted Dirac Operator \(\mathcal{D}_{\theta}\) on Riemannian Manifolds}
	
	\subsection{An Integral Inequality}
	Let \((M^{n},g)\) be a closed \(n\)-dimensional Riemannian manifold. We denote by \(\Omega^{k}(M)\) the space of real differential \(k\)-forms for \(0\le k\le n\). For notational simplicity, we often identify real vector fields with real \(1\)-forms via the musical isomorphisms. Given a smooth \(1\)-form \(\theta\in\Omega^{1}(M)\), the corresponding vector field \(\theta^{\sharp}\) on \(M\) is characterized by
	\[
	\theta(V)=g(V,\theta^{\sharp}),\qquad\forall\ \text{vector field }V\text{ on }M.
	\]
	We denote the interior product with \(\theta^{\sharp}\) by \(i_{\theta^{\sharp}}\). The Hodge star operator with respect to the metric \(g\) is denoted by \(\ast\).
	
	We define the twisted differential operator 
	\(d_{\theta}:\Omega^{k}(M)\to\Omega^{k+1}(M)\) by
	\[
	d_{\theta}=d+\theta\wedge\cdot,
	\]
	and its formal adjoint 
	\(d^{*}_{\theta}:\Omega^{k}(M)\to\Omega^{k-1}(M)\) by
	\[
	d^{*}_{\theta}=d^{*}+i_{\theta^{\sharp}} .
	\]
	The \emph{twisted Dirac operator} \(\mathcal{D}_{\theta}\) acting on differential forms is defined as
	\[
	\mathcal{D}_{\theta}=d_{\theta}+d^{*}_{\theta}:\Omega^{+}(M)\to\Omega^{-}(M),
	\]
	where 
	\[\Omega^{+}(M)=\bigoplus_{k=\text{ even}}\Omega^{k}(M)\,\quad and\,\quad
	\Omega^{-}(M)=\bigoplus_{k=\text{ odd}}\Omega^{k}(M).\]
	When \(\theta\) is a closed 1-form (\(d\theta=0\)), the operator \(d_{\theta}\) satisfies \(d_{\theta}^2=0\), 
	and one can define the \emph{Morse--Novikov cohomology} (also known as \emph{twisted de Rham cohomology}) groups
	\[
	H^k(M,\theta):=\frac{\ker d_{\theta}\cap\Omega^k(M)}{\operatorname{im}d_{\theta}\cap\Omega^k(M)}.
	\]
	This theory was introduced by Novikov \cite{Novikov1982} in the context of multi-valued functions and Hamiltonian systems, and has since found applications in symplectic geometry, topology, and mathematical physics (see \cite{deLeon2003}). For a closed 1-form \(\theta\), the twisted cohomology groups \(H^k(M,\theta)\) are finite-dimensional vector spaces that depend only on the cohomology class \([\theta]\in H^1_{\text{dR}}(M)\).
	
	The twisted Dirac operator \(\mathcal{D}_{\theta}\) is intimately related to Morse--Novikov cohomology. Indeed, its square is given by
	\[
	\mathcal{D}_{\theta}^2 = \Delta_{\theta} := d_{\theta}d_{\theta}^{*}+d_{\theta}^{*}d_{\theta},
	\]
	which is the \emph{twisted Laplacian}. The space of harmonic forms with respect to this Laplacian,
	\[
	\mathcal{H}^k_{\theta}(M):=\ker\Delta_{\theta}\cap\Omega^k(M),
	\]
	is isomorphic to \(H^k(M,\theta)\) via the Hodge theorem for twisted cohomology. In particular,
	\[
	\dim\mathcal{H}^k_{\theta}(M)=\dim H^k(M,\theta).
	\]
	When \(\theta=0\), we recover the ordinary de Rham cohomology and the Hodge theorem. For non-exact closed forms \(\theta\), the twisted cohomology groups provide a deformation of the usual cohomology that captures information about the dynamics of the multi-valued potential associated with \(\theta\).
	
	An important feature of Morse-Novikov cohomology is the \emph{Poincar\'e duality}:
	\[
	H^k(M,\theta)\cong (H^{n-k}(M,-\theta))^*,
	\]
	where \(*\) denotes the dual vector space. This will play a crucial role in our topological applications.
	
	From an index-theoretic perspective, the twisted Dirac operator \(\mathcal{D}_{\theta}\) is particularly interesting. When \(\theta\) is closed, it is an elliptic operator whose index is given by the \emph{twisted Euler characteristic}:
	\[
	\operatorname{Index}(\mathcal{D}_{\theta})=\sum_{k=0}^{n}(-1)^k\dim H^k(M,\theta).
	\]
	Remarkably, when \(M\) is oriented, this twisted Euler characteristic equals the ordinary Euler characteristic \(\chi(M)\), independent of \(\theta\). This follows from the invariance of the index under continuous deformations of the operator and the fact that \(\mathcal{D}_{\theta}\) can be continuously deformed to \(\mathcal{D}_0\) when \(\theta\) is closed. This observation provides a powerful link between the analytic properties of \(\mathcal{D}_{\theta}\) and the topology of \(M\).
	
	The following index formula is standard and will be fundamental for our results:
	
	\begin{lemma}[Index Formula] \label{lem:index-formula} 
		The index of \(\mathcal{D}_{\theta}\) satisfies
		\[
		\operatorname{Index}(\mathcal{D}_{\theta})
		:=\dim\ker(\mathcal{D}_{\theta})-\dim\operatorname{coker}(\mathcal{D}_{\theta})
		=\chi(M).
		\]
		Moreover, if \(d\theta=0\), then
		\[
		\sum_{k=0}^{n}(-1)^{k}
		\dim\mathcal{H}_{\theta}^{k}(M)
		=\chi(M).
		\]
	\end{lemma}
Similar integral identities have been established for certain Dirac operators \cite{Chen2020, Chen2024b}. In our analysis, however, we rely on the following more general inequality, which plays a fundamental role (see \cite[Theorem 3.4]{Huang2025}).
\begin{theorem}[Integral Inequality]\label{thm:integral-identity} 
	(\cite[Theorem 3.4]{Huang2025})  
	Let \((M^{n},g)\) be a closed \(n\)-dimensional smooth Riemannian manifold. 
	For every \(\alpha\in\ker\mathcal{D}_{\theta}\cap\Omega^{\pm}(M)\), we have
	\[
	\int_{M}|\theta|^{2}|\alpha|^{2}
	\le C_{2}(n)\int_{M}|\nabla\theta|\cdot|\alpha|^{2},
	\]
	where \(C_{2}(n)>0\) is a constant depending only on the dimension \(n\).
\end{theorem}

	\subsection{A Priori \(L^{\infty}\)-Estimates for Forms in \(\ker\mathcal{D}_{\theta}\)}
	
	We now derive an a priori \(L^{\infty}\)-estimate for smooth differential forms \(\alpha\in\ker\mathcal{D}_{\theta}\cap\Omega^{\pm}(M)\) via Moser iteration. 
	The analysis begins with a pointwise Bochner-type identity (see \cite[Page 181]{Petersen1998}).
	
	For any \(\alpha\in\Omega^{l}(M)\),
	\begin{equation}\label{eq:bochner-basic}
	-\frac{1}{2}\Delta_{d}|\alpha|^{2}
	=|\nabla\alpha|^{2}-\langle\nabla^{*}\nabla\alpha,\alpha\rangle
	=|\nabla\alpha|^{2}-\langle\Delta_{d}\alpha,\alpha\rangle
	+\langle\operatorname{Ric}(\alpha),\alpha\rangle,
	\end{equation}
	where \(\operatorname{Ric}(\alpha)\) denotes the Weitzenb\"ock curvature operator acting on \(\alpha\).
	
	Since the operator \[d+d^{*}:\Omega^{+}(M)\to\Omega^{-}(M)\] is a Dirac operator (see \cite{Lawson1989}) and \[\Delta_{d}=(d+d^{*})^{2},\] 
	we can apply a key observation from \cite{Chen2024b} (also \cite[Page 115, Proposition 5.3]{Lawson1989}). For any smooth differential form \(\alpha\),
	\begin{equation}\label{eq:div-representation}
	\langle\Delta_{d}\alpha,\alpha\rangle
	=\langle(d+d^{*})\alpha,(d+d^{*})\alpha\rangle+\operatorname{div} X,
	\end{equation}
	where \(X\) is a vector field defined by
	\begin{equation}\label{eq:V-definition}
	\langle X,W\rangle
	=-\langle(d+d^{*})\alpha,W^{\sharp}\wedge\alpha+i_{W}\alpha\rangle,
	\qquad\forall\ \text{vector field }W.
	\end{equation}
	\begin{proposition}[Bochner Inequality for \(\ker\mathcal{D}_{\theta}\)] \cite[Proposition 3.7]{Huang2025}\label{prop:bochner-kerD}
		Let \((M^{n},g)\) be a closed \(n\)-dimensional Riemannian manifold and let \(\theta\) be a smooth \(1\)-form on \(M\). 
		Fix an integer \(1 \leq l \le \lfloor\frac{n}{2} \rfloor\). For every \(\alpha \in \ker\mathcal{D}_{\theta} \cap \Omega^{\pm}(M)\), the following Bochner-type inequality holds:
		\begin{equation}\label{eq:bochner-ineq}
		-\frac{1}{2}\Delta_{d}|\alpha|^{2}
		\ge |\nabla\alpha|^{2} - |\theta|^{2}|\alpha|^{2}
		- C_{3}(n)\lambda^{-}_{n,l}|\alpha|^{2} - \operatorname{div} X,
		\end{equation}
			where $C_{3}(n):=\lfloor\frac{n}{2}\rfloor$,  \(\lambda^{-}_{n,l} = \max\{0, -\lambda_{n,l}\}\), 
		and the vector field \(X\) satisfies the pointwise bound
		\[
		|X| \le |\theta| \, |\alpha|^{2}.
		\]
	\end{proposition}
	
	\begin{proof}
		The condition \(\alpha\in\ker\mathcal{D}_{\theta}\) implies $(d_{\theta}+d^{*}_{\theta})\alpha = 0.$ 	A direct computation gives
		\[
		(d+d^{*})\alpha = -\theta\wedge\alpha - i_{\theta^{\sharp}}\alpha,
		\]
		and consequently
		\[
		|(d+d^{*})\alpha| = |\theta|\,|\alpha|.
		\]
		From the identities \eqref{eq:bochner-basic} and \eqref{eq:div-representation} we obtain
		\[
		-\frac{1}{2}\Delta_{d}|\alpha|^{2}
		= |\nabla\alpha|^{2}
		- \langle (d+d^{*})\alpha, (d+d^{*})\alpha\rangle
		+ \langle \operatorname{Ric}(\alpha), \alpha\rangle
		- \operatorname{div} X.
		\]
		Applying the curvature estimate \eqref{eq:curvature-bound} from Corollary \ref{cor:eigenvalue-estimate} yields
		\begin{align*}
		\langle \operatorname{Ric}(\alpha), \alpha \rangle
		&\ge \frac{1}{n-l}\Bigl\lfloor\frac{n}{2}\Bigr\rfloor
		\Bigl(n - \Bigl\lfloor\frac{n}{2}\Bigr\rfloor\Bigr)\lambda_{n,l}|\alpha|^{2} \\
		&\ge -\frac{1}{n-l}\Bigl\lfloor\frac{n}{2}\Bigr\rfloor
		\Bigl(n - \Bigl\lfloor\frac{n}{2}\Bigr\rfloor\Bigr)\lambda^{-}_{n,l}|\alpha|^{2} \\
		&\ge -\Bigl\lfloor\frac{n}{2}\Bigr\rfloor\lambda^{-}_{n,l}|\alpha|^{2}.
		\end{align*}
		Substituting \(|(d+d^{*})\alpha|^{2} = |\theta|^{2}|\alpha|^{2}\) together with this estimate into the identity completes the proof.
	\end{proof}
	We now apply the mean value inequality (Proposition \ref{prop:mean-value}) to obtain uniform bounds.
	
	\begin{theorem}[\(L^{\infty}\)-Estimate]\label{thm:Linfty-estimate}
		Let \((M^{n},g)\) be a closed \(n\)-dimensional smooth Riemannian manifold satisfying the Sobolev inequality
		\[
		\|f\|_{\frac{2p}{p-2}}\le\|f\|_{2}+C_{s}D\|df\|_{2},\qquad p>2.
		\]
		For each \(\alpha\in\ker\mathcal{D}_{\theta}\cap\Omega^{\pm}(M)\) and {\(+\infty\geq q>p\)}, we have
		{ \begin{equation}\label{eq:Linfty-bound-alpha}
\|\alpha\|_{\infty}
		\le C_1(p,q)\bigl(1+\sqrt{C_{0}(p,q)C_{s}DB}\bigr)^{\frac{2pq}{q-p}}\|\alpha\|_{2},
		\end{equation}	}
	where \(B^{2}:=C_3(n) \bigl(\|\theta\|_{q}^{2}+\|\la^{-}_{n,l}\|_{\frac{q}{2}}\bigr)\).
	\end{theorem}
	
	\begin{proof}
		From Proposition \ref{prop:bochner-kerD}, we have the differential inequality
		\[
		-\frac{1}{2}\Delta_{d}|\alpha|^{2}
		\ge-|\theta|^{2}|\alpha|^{2}-C_3(n)\la^{-}_{n,l}|\alpha|^{2}-\operatorname{div}X.
		\]
		Applying the Kato inequality \(|\nabla|\alpha||\le|\nabla\alpha|\) (see \cite{Berard1988}), we obtain
		\[
		-|\alpha|\Delta_{d}|\alpha|
		\ge-|\theta|^{2}|\alpha|^{2}-C_3(n)\la^{-}_{n,l}|\alpha|^{2}-\operatorname{div}X.
		\]    
			Now set \(f=|\alpha|\). Then \(f\) satisfies
		\[
		f\Delta_{d}f\le\bigl(|\theta|^{2}+C_3(n)\la^{-}_{n,l}\bigr)f^{2}+\operatorname{div}X,
		\]
		with \(|X|\le|\theta|f^{2}\). Define
		\[
		h_{1}=|\theta|^{2}+C_3(n)\la^{-}_{n,l},\qquad h_{2}=|\theta|.
		\]
	   Their norms satisfy
	\[
	\|h_{1}\|_{q}\le \|\theta\|_{q}^{2}+C_{3}(n)\|\lambda^{-}_{n,l}\|_{\frac{q}{2}},
	\qquad
	\|h_{2}\|_{q}= \|\theta\|_{q}.
	\]
	Applying Proposition \ref{prop:mean-value} with \(h_{1},h_{2}\) as above gives precisely the estimate \eqref{eq:Linfty-bound-alpha}.
	\end{proof}
	
	\subsection{\(L^{\infty}\)-Estimates for Solutions of Schr\"odinger-Type Equations}
We consider Schrödinger operators of the form 
\[
\overline{\Delta} + V,
\]
where the potential \(V \in C^\infty\bigl(\operatorname{Sym}(E)\bigr)\) is a smooth field of symmetric endomorphisms of \(E\). By the Weitzenböck formula, the Hodge Laplacian \(\Delta_{d}\) acting on \(1\)-forms is precisely such an operator.

The following estimate for sections satisfying a Schrödinger equation will play a crucial role when we later apply our results to harmonic \(1\)-forms.
	\begin{proposition}[\(L^{\infty}\)-estimate for solutions to Schrödinger-type equations]\label{prop:schrodinger-Linfty}
		Let \((M^{n},g)\) be a closed \(n\)-dimensional Riemannian manifold satisfying the Sobolev inequality
		\[
		\|f\|_{\frac{2p}{p-2}} \le \|f\|_{2} + C_{s}D \|df\|_{2}, \qquad p>2.
		\]
		Assume \(+\infty \ge q > p\). Then every solution of
		\[
		\nabla^{*}\nabla S + V S = 0
		\]
		satisfies the \(L^\infty\)-estimate
		\begin{equation}\label{eq:schrodinger-bound}
		\|S\|_{\infty}
		\le \exp\Bigl\{ C_4(p,q)\,\sqrt{\|V^{-}\|_{\frac{q}{2}}}\,C_{s}D \Bigr\} \, \|S\|_{2},
		\end{equation}
		where \[C_4(p,q):= \frac{2p}{p-2}\sqrt{C_{0}(p,q)} 
		\frac{1}{\sqrt{\gamma}(\sqrt{\gamma}-1)}. \]
	\end{proposition}
	
	\begin{proof}
		Applying the Kato inequality to \(\nabla^{*}\nabla S + V S = 0\) yields
		\[
		|S|\,\Delta_{d}|S| \le \langle\nabla^{*}\nabla S,S\rangle \le V^{-}|S|^{2}.
		\]
		Hence \(f := |S|\) satisfies
		\[
		f\Delta_{d}f \le |V^{-}| f^{2}.
		\]
		Proceeding as in the derivation of \eqref{V6} (with \(h_{1}=|V^{-}|\) and \(h_{2}=0\)), for any \(k\ge 1\) we obtain
		\begin{align*}
		\int_M f^{2k-2}|\nabla f|^2 
		&\le \frac{2}{3k-1}\, \|V^{-}\|_{L^{\frac{q}{2}}} \|f\|_{L^{\frac{2kq}{q-2}}}^{2k} \\
		&\le \frac{1}{k}\, \|V^{-}\|_{L^{\frac{q}{2}}} \|f\|_{L^{\frac{2kq}{q-2}}}^{2k}.
		\end{align*}
		Consequently,
		\[
		\|d f^{k}\|_2 \le k^{\frac{1}{2}}\,\|V^{-}\|_{\frac{q}{2}}^{\frac{1}{2}} \|f\|_{\frac{2kq}{q-2}}^{k}.
		\]
		A standard Moser iteration (similar to, but simpler than, the proof of Proposition~\ref{prop:mean-value}) now gives
		\begin{align*}
		\frac{\|f\|_\infty^2}{\|f\|_2^2}
		&\le \prod_{\ell=0}^{\infty} 
		\Bigl(1 + C_s D \sqrt{C_{0}(p,q)}\,\|V^{-}\|_{\frac{q}{2}}^{\frac{1}{2}} \gamma^{\frac{\ell}{2}} \Bigr)^{\frac{2p}{(p-2)\gamma^{\ell+1}}} \\
		&\le \exp\Bigg( \sum_{\ell=0}^{\infty} 
		\frac{2p}{(p-2)\gamma^{\ell+1}}
		\ln\Bigl(1 + C_s D \sqrt{C_{0}(p,q)}\|V^{-}\|_{\frac{q}{2}}^{\frac{1}{2}} \gamma^{\frac{\ell}{2}} \Bigr) \Bigg) \\
		&\le \exp\Bigg( \frac{2p}{p-2}C_s\sqrt{C_{0}(p,q)} \sqrt{\|V^{-}\|_{\frac{q}{2}}D^{2}}\;
		\frac{1}{\sqrt{\gamma}(\sqrt{\gamma}-1)} \Bigg).
		\end{align*}
		Taking square roots and using \(f = |S|\) yields the desired estimate.
	\end{proof}
	The next lemma connects the Sobolev--Poincar\'e inequality to the ordinary Sobolev inequality, which will be needed in the sequel.	
	\begin{lemma}\label{lem:sobolev-poincare-implies-sobolev}
		Let \((M^{n},g)\) be a closed \(n\)-dimensional (\(n\ge4\)) smooth Riemannian manifold satisfying the Sobolev--Poincar\'e inequality
		\[
		\|f-\bar{f}\|_{2}\le C_{s}D\|df\|_{\frac{2p}{p+2}},
		\]
		where \(\bar{f}=\frac{1}{\operatorname{Vol}(g)}\int_{M}f\,dvol_{g}\) and $p>2$. Then
		\[
		\|f\|_{\frac{2p}{p-2}}
		\le\|f\|_{2}+C_{s}\frac{p}{p-2}D\|df\|_{2}.
		\]
	\end{lemma}
	
	\begin{proof}
		Let \(u=|f|^{\frac{p}{p-2}}\). Then
		\[
		\bar{u}=\frac{1}{\operatorname{Vol}(g)}\int_{M}|f|^{\frac{p}{p-2}}dvol_{g}
		=\|f\|_{\frac{p}{p-2}}^{\frac{p}{p-2}},
		\]
		and
		\[
		|du|=\frac{p}{p-2}|d|f||\cdot|f|^{\frac{2}{p-2}}
		\le\frac{p}{p-2}|df|\cdot|f|^{\frac{2}{p-2}}.
		\]    
		By Sobolev--Poincar\'{e} inequality, H\"{o}lder inequality and Gagliardo--Nirenberg interpolation inequality, we have
		\begin{align*}
		\|f\|^{\frac{p}{p-2} }_{\frac{2p}{p-2} }=\|u\|_{2}
		&\leq\|u-\bar{u}\|_{2}+|\bar{u}|\\
		&\leq C_{s}\frac{p}{p-2}D\|d|f|\cdot|f|^{\frac{2}{p-2}}\|_{\frac{2p}{p+2}}+|\bar{u}|\\
		&\leq C_{s}\frac{p}{p-2}D\|df\|_{2}\|f\|^{\frac{2}{p-2} }_{\frac{2p}{p-2}}+\|f\|^{\frac{p}{p-2} }_{\frac{p}{p-2}}\\
		&\leq C_{s}\frac{p}{p-2}D\|df\|_{2}\|f\|^{\frac{2}{p-2} }_{\frac{2p}{p-2}}+\|f\|_{2}\|f\|^{\frac{2}{p-2} }_{\frac{2p}{p-2}}.
		\end{align*}
		Dividing both sides by \(\|f\|_{\frac{2p}{p-2}}^{\frac{2}{p-2}}\) yields the desired inequality.
	\end{proof}

	\subsection{Key Estimates for Forms in \(\ker\mathcal{D}_{\theta}\)}

	Using the integral identity (Theorem \ref{thm:integral-identity}) and the \(L^{\infty}\)-estimates established above, we now prove a technical lemma that are essential for the proof of our main theorem.
	
	\begin{lemma}\label{lem:key-estimates}
		Let \((M^{n},g)\) be a closed \(n\)-dimensional smooth Riemannian manifold satisfying the Sobolev--Poincar\'e inequality
		\[
		\|f-\bar{f}\|_{2}\le C_{s}D\|df\|_{\frac{2p}{p+2}},\qquad p>2.
		\]		
		Suppose that \(\theta\) is a \(1\)-form satisfying
		\[
		\nabla^{*}\nabla\theta+V\theta=0.
		\]
		Then for each \(\alpha\in\ker\mathcal{D}_{\theta}\cap\Omega^{\pm}(M)\), we have	
		\begin{align}
		\frac{1}{\operatorname{Vol}(g)}\int_{M}|u-\bar{u}|\,|\alpha|^{2}
		&\le 2C_{s}C^2_1(p,q)\sqrt{\|V^{-}\|_{\frac{q}{2}}D^{2}}\bigl(1+\sqrt{C_{0}(p,q)C_{s}DB}\bigr)^{\frac{4pq}{q-p}}\|\alpha\|_{2}^{2}\nonumber\\
		&\quad\times\exp\bigl(2C_4(p,q)\sqrt{\|V^{-}\|_{\frac{q}{2}}}C_{s}D\bigr)\|\theta\|_{2}^{2},\label{eq:mean-oscillation-estimate}\\
		\frac{1}{\operatorname{Vol}(g)}\int_{M}|\theta|^{2}|\alpha|^{2}
		&\le C_{2}(n) C^2_1(p,q) \sqrt{\|V^{-}\|_{\frac{q}{2}}} \bigl(1+\sqrt{C_{0}(p,q)C_{s}DB}\bigr)^{\frac{4pq}{q-p}}\|\alpha\|_{2}^{2}\nonumber\\
		&\quad\times\exp\bigl(C_4(p,q)\sqrt{\|V^{-}\|_{\frac{q}{2}}}C_{s}D\bigr)\|\theta\|_{2},\label{eq:theta-square-estimate}
		\end{align}
		where 
		\(u=|\theta|^{2}\), 
		\(\bar{u}=\frac{1}{\operatorname{Vol}(g)}\int_{M}|\theta|^{2}dvol_{g}\) and $+\infty\geq q>p$.
	\end{lemma}	
	\begin{proof}
		First note that from the equation \(\nabla^{*}\nabla\theta+V\theta=0\), we have
		\begin{align*}
		\|\nabla\theta\|_{2}^{2}
		&=\frac{1}{\operatorname{Vol}(g)}\int_{M}\langle\nabla^{*}\nabla\theta,\theta\rangle dvol_{g}\\
		&\le\frac{1}{\operatorname{Vol}(g)}\int_{M}V^{-}|\theta|^{2}dvol_{g}\\
		&\le\|V^{-}\|_{\frac{q}{2}}\|\theta\|_{\infty}^{2}.
		\end{align*}
		For \eqref{eq:mean-oscillation-estimate}, we compute
		\begin{align*}
		\frac{1}{\operatorname{Vol}(g)}\int_{M}|u-\bar{u}|\,|\alpha|^{2}
		&\le\|u-\bar{u}\|_{2}\|\alpha\|_{4}^{2} \\
		&\le C_{s}D\|\na u\|_{\frac{2p}{p+2} }\|\a\|^{2}_{4}\\
		&\le C_{s}D\|\nabla u\|_{2}\|\alpha\|_{\infty}^{2} \\
		&\le 2C_{s}D\|\theta\|_{\infty}\|\nabla\theta\|_{2}\|\alpha\|_{\infty}^{2} \\
		&\le 2C_{s}\sqrt{\|V^{-}\|_{\frac{q}{2}}D^{2}}\|\theta\|_{\infty}^{2}\|\alpha\|_{\infty}^{2}.
		\end{align*}    
		Now apply Theorem \ref{thm:Linfty-estimate} to bound \(\|\alpha\|_{\infty}\) and Proposition \ref{prop:schrodinger-Linfty} to bound \(\|\theta\|_{\infty}\). This yields \eqref{eq:mean-oscillation-estimate}.
		
		For \eqref{eq:theta-square-estimate}, we use Theorem \ref{thm:integral-identity}:
		\begin{align*}
		\frac{1}{\operatorname{Vol}(g)}\int_{M}|\theta|^{2}|\alpha|^{2}
		&\le C_{2}(n)\frac{1}{\operatorname{Vol}(g)}\int_{M}|\nabla\theta|\,|\alpha|^{2} \\
		&\le C_{2}(n)\|\nabla\theta\|_{2}\|\alpha\|_{\infty}^{2} \\
		&\le C_{2}(n)\sqrt{\|V^{-}\|_{\frac{q}{2}}}\|\theta\|_{\infty}\|\alpha\|_{\infty}^{2}.
		\end{align*}
		Again applying the \(L^{\infty}\)-estimates gives \eqref{eq:theta-square-estimate}.
	\end{proof}
	
	We now combine the estimates from the previous section to prove our first main theorem.
	
	\begin{proof}[\textbf{Proof of Theorem \ref{thm:main1}}]
		Assume the hypotheses of Theorem 1.1 hold. In particular, we have
		\[
		\sqrt{\|V^{-}\|_\frac{q}{2} D^2} \le \varepsilon,
		\]
		where \(\varepsilon\) will be determined by \eqref{V4},  \eqref{V3} and \eqref{V1}.  
			
First, observe that if
\begin{equation}\label{V4}
\sqrt{\|V^{-}\|_\frac{q}{2} D^2} \le \frac{1}{2C_s C_{4}(p,q)},
\end{equation}
we then have
\begin{equation}\label{V2}
e^{2C_{4}(p,q)\sqrt{\|V^{-}\|_{\frac{q}{2}}}C_{s}D}\leq e.
\end{equation}    
From Proposition \ref{prop:schrodinger-Linfty} we obtain
\begin{equation}\label{eq:theta-infinity-bound}
\|\theta\|^{2}_{q}\leq\|\theta\|^{2}_\infty \le e \cdot \|\theta\|^{2}_2.
\end{equation}    
Now choose \(t > 0\) such that
		\begin{equation}\label{eq:t-choice}
		\|t\theta\|_2^2 = 9e^4 C_2^2(n) C^2_1(p,q) \|\la^{-}_{n,l}\|_\frac{q}{2}.
		\end{equation}    
		With this choice, we have
		\begin{align*}
		B^2 &:= C_3(n)(\|t\theta\|_{q}^2 +\|\la^{-}_{n,l}\|_\frac{q}{2}) \\
		&\le C_{3}(n)\big{(}9e^5 C_2^2(n) C^2_1(p,q)+1\big{)} \|\la^{-}_{n,l}\|_\frac{q}{2}\\
		&:=C_5(n,p,q) \|\la^{-}_{n,l}\|_\frac{q}{2}. 
		\end{align*}    
Noting that
	\begin{align*}
		\bigl(1 + \sqrt{C_0(p,q)C_s D B}\bigr)^{\frac{4pq}{q-p}}
		&\le \exp\Bigl(\frac{4pq}{q-p}\sqrt{C_0(p,q)C_s D B}\Bigr) \\
		&\le \exp\Bigl(\frac{4pq}{q-p} \sqrt{C_0(p,q)C_s}(C_{5}(n,p,q)\|\la^{-}_{n,l}\|_{\frac{q}{2}}D^{2})^{\frac{1}{4}} \Bigr)\\
		&:=\exp\Bigl(\sqrt{C_s}C_6(n,p,q)(\|\la^{-}_{n,l}\|_{\frac{q}{2}}D^{2})^{\frac{1}{4}}\Bigr ).
		\end{align*}	
Now suppose further 
\begin{equation}	\label{V3}
\sqrt{\|V^{-}\|_\frac{q}{2} D^2} \le e\sqrt{\|\la^{-}_{n,l}\|_{\frac{q}{2}}D^{2}}\exp\Bigl(-\sqrt{C_s}C_6(n,p,q)(\|\la^{-}_{n,l}\|_{\frac{q}{2}}D^{2})^{\frac{1}{4}}\Bigr ),
\end{equation}				
Applying estimate \eqref{eq:theta-square-estimate} from Lemma \ref{lem:key-estimates} to \(t\theta\) and using \eqref{eq:t-choice}--\eqref{V3}, we obtain
		\begin{align*}
		\frac{1}{\operatorname{Vol}(g)} &\int_M t^2|\theta|^2 |\alpha|^2 \\
		&\le C_2(n) \sqrt{\|V^{-}\|_\frac{q}{2}} C_1(p,q) 
		\bigl(1 + \sqrt{C_0(p,q)C_s D B}\bigr)^{\frac{4pq}{q-p}}
		\|\alpha\|_2^2 e^{2C_sC_4(p,q) D\sqrt{\|V^{-}\|_\frac{q}{2} }} \|t\theta\|_2 \\
		&\le e C_2(n) C_1(p,q) \sqrt{\|V^{-}\|_\frac{q}{2}}	\exp\Bigl(\sqrt{C_s}C_6(n,p,q)(\|\la^{-}_{n,l}\|_{\frac{q}{2}}D^{2})^{\frac{1}{4}}\Bigr )\|\alpha\|_2^2 \|t\theta\|_2 \\
		&= 3e^3 C_2^2(n) C^2_1(p,q)
		\Bigl(\frac{\|V^{-}\|_\frac{q}{2}}{\|\la^{-}_{n,l}\|_\frac{q}{2}}\Bigr)^{\frac{1}{2}}\|\la^{-}_{n,l}\|_\frac{q}{2}\exp\Bigl(\sqrt{C_s}C_6(n,p,q)(\|\la^{-}_{n,l}\|_{\frac{q}{2}}D^{2})^{\frac{1}{4}}\Bigr )
		\|\alpha\|_2^2 \\
		&= \frac{1}{3e} 
		\Bigl(\frac{\|V^{-}\|_\frac{q}{2}}{\|\la^{-}_{n,l}\|_\frac{q}{2}}\Bigr)^{\frac{1}{2}}\exp\Bigl(\sqrt{C_s}C_6(n,p,q)(\|\la^{-}_{n,l}\|_{\frac{q}{2}}D^{2})^{\frac{1}{4}}\Bigr )
		\|t\theta\|_2^2 \|\alpha\|_2^2\\
		&\leq \frac{1}{3}
		\|t\theta\|_2^2 \|\alpha\|_2^2.
		\end{align*}    
If, in addition, 
		\begin{equation}	\label{V1}
		\sqrt{\|V^{-}\|_\frac{q}{2} D^2} \le \frac{1}{6\sqrt{e}C_s C_{1}(p,q)}\exp\Bigl(-\sqrt{C_s}C_6(n,p,q)(\|\la^{-}_{n,l}\|_{\frac{q}{2}}D^{2})^{\frac{1}{4}}\Bigr ),
		\end{equation}	
	then by the estimate \eqref{eq:mean-oscillation-estimate}, we get
		\begin{align*}
		\frac{1}{\operatorname{Vol}(g)} &\int_M t^2 |u - \bar{u}| |\alpha|^2 \\
		&\le 2C_s C_1(p,q) \sqrt{\|V^{-}\|_\frac{q}{2} D^2}
		\bigl(1 + \sqrt{C_0(p,q)C_s D B}\bigr)^{\frac{4pq}{q-p}}
		\|\alpha\|_2^2 e^{C_sC_4(p,q)D\sqrt{\|V^{-}\|_\frac{q}{2}  }} \|t\theta\|_2^2 \\
		&\le 2\sqrt{e} C_s C_1(p,q) 
		\sqrt{\|V^{-}\|_\frac{q}{2} D^2}\exp\Bigl(\sqrt{C_s}C_6(n,p,q)(\|\la^{-}_{n,l}\|_{\frac{q}{2}}D^{2})^{\frac{1}{4}}\Bigr ) \|\alpha\|_2^2 \|t\theta\|_2^2 \\
		&\le \frac{1}{3} \|t\theta\|_2^2 \|\alpha\|_2^2.
		\end{align*}    
		Now observe that by the triangle inequality,
			\begin{align*}
		\|t\theta\|_2^2 \|\alpha\|_2^2 
	&	= \frac{1}{\operatorname{Vol}(g)} \int_M t^2 \bar{u} |\alpha|^2\\
		&\le \frac{1}{\operatorname{Vol}(g)} \int_M t^2 |u - \bar{u}| |\alpha|^2
		+ \frac{1}{\operatorname{Vol}(g)} \int_M t^2 |\theta|^2 |\alpha|^2,
			\end{align*}   
	where $u=|\theta|^{2}$ and $\bar{u}=\frac{1}{\operatorname{Vol}(g)}\int_{M}|\theta|^{2}dvol_{g}$.
	
		Combining the preceding estimates, we obtain
		\[
		\|t\theta\|_2^2 \|\alpha\|_2^{2}\leq \frac{2}{3} \|t\theta\|_2^2 \|\alpha\|_2^2. \]
		This implies \(\|\alpha\|_2 = 0\), hence \(\alpha = 0\). Therefore,
		\[
		\ker\mathcal{D}_{t\theta}\cap\Omega^{\pm}_{l}(M)=0.
		\]
		Finally, if  $l=\lfloor\frac{n}{2}\rfloor$, then $\Om^{\pm}_{l}=\Om^{\pm}$. By Lemma \ref{lem:index-formula},
		\begin{equation*}
		\begin{split}
		{\rm{Index}}(\mathcal{D}_{t\theta}):&=\dim\ker(\mathcal{D}_{t\theta})-\dim\operatorname{coker}(\mathcal{D}_{t\theta})\\
		&=\dim\ker(\mathcal{D}_{t\theta})\cap\Om^{+}-\dim\operatorname{ker}(\mathcal{D}_{t\theta})\cap\Om^{-}\\
		&=\chi(M).
		\end{split}
		\end{equation*}  
		It implies that \(\chi(M) = 0\).
	\end{proof}

\section{Applications to Integral Curvature Conditions}

\subsection{Curvature Operator with \(L^p\) Lower Bound}
	
	We first establish that under suitable integral Ricci curvature bounds, the manifold satisfies a Sobolev--Poincar\'e inequality with explicit constants.
	
	\begin{lemma}[Sobolev--Poincar\'e Inequality under Integral Ricci Bounds]\label{lem:sobolev-poincare-ricci}
		Let \((M^{n},g)\) be a closed \(n\)-dimensional smooth Riemannian manifold, and let $q>p>n$.  There exists a positive constant \(C_7(n,p)\) such that, if the Ricci curvature satisfies
		\[
		\|\operatorname{Ric}^{-}\|_\frac{q}{2} D^2 \le C_7(n,p),
		\]
		then for every function \(f \in L^2_1(M)\), we have the Sobolev--Poincar\'e inequality
		\[
		\|f - \bar{f}\|_2 \le C_8(n,p)D\|df\|_{\frac{2p}{p+2}},
		\]
		where \(\bar{f} = \frac{1}{\operatorname{Vol}(g)} \int_M f\, dvol_g\).
	\end{lemma}
	\begin{proof}
        Recall the Sobolev constant is given by (\cite[Page 201]{Gallot1988}) 
		\[
		\lambda_{\frac{2p}{p+2},2} 
		= \inf_{f \in C^\infty(M)} 
		\frac{\|\nabla f\|_{\frac{2p}{p+2}}}{\inf_{a \in \mathbb{R}} \|f - a\|_2 }.
		\]
		The constants occur in the embedding $L^{\frac{2p}{p+2}}_{1}\hookrightarrow L^{2}$ for any $p>n$.
		
		Since the infimum over constants \(a\) is attained at the mean value \(\bar{f}\), i.e.
		$$\inf_{a\in\mathbb{R}} \|f-a\|_{2}=\|f-\bar{f}\|_{2},$$ we have
		\[
			\|\na f\|_{\frac{2p}{p+2}}\geq \la_{\frac{2p}{p+2},2}  \|f-\bar{f}\|_{2}.\]
		Now observe that
		\[
		\Bigl(\frac{1}{\operatorname{Vol}(g)} \int_M (\operatorname{Ric}^{-} D^2 - 1)_+^\frac{p}{2} \, dvol_g\Bigr)^{\frac{2}{p}}
		\le \|\operatorname{Ric}^{-}\|_\frac{p}{2} D^2\leq \|\operatorname{Ric}^{-}\|_\frac{q}{2} D^2,
		\]
		where \((\operatorname{Ric}^{-} D^2 - 1)_+ = \max\{\operatorname{Ric}^{-} D^2 - 1, 0\}\).
		
	By \cite[Theorem 6]{Gallot1988} (in this time $\a=\frac{1}{D}$),  if \[\|\operatorname{Ric}^{-}\|_\frac{q}{2} D^2 \le C_{7}(n,p):=\frac{1}{2}(e^{B(p)}-1)^{-1},\]
		then there exists a positive constant
		\[
		\gamma(D,n,p) := \frac{B(p)}{D} 
		\min\Bigl\{2^{-\frac{1}{p-1}}, \frac{1}{4}(e^{B(p)} - 1)^{-1}\Bigr\}
		\]
		such that
		\[
	K(p,\frac{2p}{p+2})\cdot	\lambda_{\frac{2p}{p+2},2} \ge \gamma(D,n,p),
		\]
		where \(B(p)\) is the constant appearing in \cite[Theorem 2]{Gallot1988} and $K(p,\frac{2p}{p+2})$ is the constant appearing in \cite[Theorem 6]{Gallot1988}. This immediately implies the desired inequality with 
		\[C_8(n,p)D = \gamma(D,n,p)^{-1}K(p,\frac{2p}{p+2}).\]
		\end{proof}
As an immediate corollary, we obtain uniform $L^\infty$-estimates for harmonic forms under an integral curvature operator condition. Building on the work of Li \cite{Li1980}, we further extend the Bochner technique and derive bounds for the Betti numbers in this setting.	
\begin{proof}[\textbf{Proof of Theorem \ref{thm:harmonic-Linfty}}]
	Because \(\theta\) is a harmonic \(k\)-form, the Weitzenböck identity gives
	\[
	\nabla^{*}\nabla\theta + \operatorname{Ric}(\theta)=0.
	\]
	Hence the function \(f:=|\theta|\) satisfies the differential inequality
	\[
	f\,\Delta_{d}f \le \frac{k(n-k)}{n-l}\lambda^{-}_{n,l}\,f^{2}
	\le \frac{n}{4}\lambda^{-}_{n,l}\,f^{2}.
	\]
  Now suppose that \(C_{9}\) satisfies
	\[
	C_{9}\le \frac{n-l}{n-1}\,C_{7}(n,p).
	\]
	Then we have \[\|\operatorname{Ric}^{-}\|_{\frac{q}{2}}D^{2}\le  \|\frac{n-1}{n-l}\,\lambda_{n,l}^{-}\|_{\frac{q}{2} }D^{2}\le C_{7}(n,p).\]
	Lemma~\ref{lem:sobolev-poincare-ricci} therefore provides the Sobolev--Poincaré inequality
	\[
	\|f-\bar{f}\|_{2}\le C_{8}(n,p)D\|df\|_{\frac{2p}{p+2}},
	\]
	Applying Proposition~\ref{prop:schrodinger-Linfty} (with \(V^{-}\leq \frac{n}{4}\lambda^{-}_{n,l}\)) yields
	\[
	\|\theta\|_{\infty}
	\le \exp\Bigl\{ C_{4}\sqrt{\frac{n}{4}}\;C_{8}\sqrt{C_{9}} \Bigr\} \,\|\theta\|_{2}.
	\]
	The bound on the Betti number \(b_{k}(M)\) is then obtained by a standard dimension argument once \(\varepsilon(n,p,q)\) is taken sufficiently small.
\end{proof}	
	
\subsection{Vanishing of Morse--Novikov cohomlogy}
	We now apply Theorem \ref{thm:main1} to obtain constraints on the Euler characteristic under geometric curvature conditions.	

\begin{corollary}\label{cor:vanishing-kernel}
	Let \((M^{n},g)\) be a closed \(n\)-dimensional smooth Riemannian manifold with non-trivial first de~Rham cohomology \(H^{1}_{\mathrm{dR}}(M)\neq 0\), and let \(q > p > n\).
	Suppose the curvature operator satisfies
	\[
	\|\lambda_{n,l}^{-}\|_{\frac{q}{2}} \, D^{2} \le C_{10}(n,p,q),
	\]
	where
	\[
	C_{10}(n,p,q):=
	\min\!\Bigl\{
	\frac{1}{16\,C^{2}_{8}(n,p)\,C^{4}_{6}(n,p,q)},\;
	\frac{1}{72\,e^{2}\,C^{2}_8(n,p)\, C^{2}_{1}(p,q)},\;
	\frac{1}{8\,C_8^{2}(n,p)\, C^{2}_{4}(p,q)},\;
	\frac{1}{2}\,C_{7}(n,p)
	\Bigr\}.
	\]
	Then there exists a positive number \(t \in \mathbb{R}^{+}\) such that for every \(k \le l\),
	\[
	H^{k}(M,t\theta) = \{0\},
	\]
	where \(\theta\) is a non-zero harmonic \(1\)-form.
\end{corollary}

\begin{proof}
	Remark~\ref{Rem} provides the pointwise estimate
	\[
	\operatorname{Ric}^{-} \le \frac{n-1}{n-l}\,\lambda_{n,l}^{-} < 2\lambda_{n,l}^{-}.
	\]
	Consequently,
	\begin{equation}\label{eq:Ric-bound}
	\sqrt{\|\operatorname{Ric}^{-}\|_{\frac{q}{2}}D^{2}}
	\le \sqrt{2}\,\sqrt{\|\lambda^{-}_{n,l}\|_{\frac{q}{2}}D^{2}} .
	\end{equation}
Assume first that the constant \(C_{10}\) satisfies
	\[
	C_{10} \le \frac{1}{2}\,C_{7}(n,p).
	\]
	Then Lemma~\ref{lem:sobolev-poincare-ricci} yields the Sobolev--Poincar\'e inequality with constant \(C_{s}=C_{8}(n,p)\).
	
	Since \(b_{1}(M)\ge 1\), we may choose a non-trivial harmonic \(1\)-form \(\theta\).  
	It satisfies the Weitzenb\"ock equation
	\[
	\nabla^{*}\nabla\theta + \operatorname{Ric}(\theta) = 0,
	\]
	which is a Schr\"odinger-type equation as in Theorem~\ref{thm:main1} with potential \(V=\operatorname{Ric}\).
	
	If we additionally impose
	\[
	\sqrt{\|\lambda_{n,l}^{-}\|_{\frac{q}{2}}D^{2}} \le \frac{1}{4 C_{8}(n,p)C^{2}_{6}(n,p,q)},
	\]
	we obtain
	\[
	e^{-\frac{1}{2}} \le \exp\Bigl(-\sqrt{C_s}\,C_6(n,p,q)\bigl(\|\lambda^{-}_{n,l}\|_{\frac{q}{2}}D^{2}\bigr)^{\frac{1}{4}}\Bigr),
	\]
	and
	\[
	\frac{1}{6e\,C_8(n,p) C_{1}(p,q)}
	\le \frac{1}{6\sqrt{e}\,C_s C_{1}(p,q)}
	\exp\Bigl(-\sqrt{C_s}\,C_6(n,p,q)\bigl(\|\lambda^{-}_{n,l}\|_{\frac{q}{2}}D^{2}\bigr)^{\frac{1}{4}}\Bigr).
	\]
	
	Consequently, by \eqref{eq:Ric-bound} we get
	\begin{align}\label{eq:cond1}
	\sqrt{\|\operatorname{Ric}^{-}\|_{\frac{q}{2}}D^{2}}
	&\le \sqrt{2}\,\sqrt{\|\lambda^{-}_{n,l}\|_{\frac{q}{2}}D^{2}} \nonumber \\
	&< e^{\frac{1}{2}}\sqrt{\|\lambda^{-}_{n,l}\|_{\frac{q}{2}}D^{2}} \nonumber \\
	&\le e\sqrt{\|\lambda^{-}_{n,l}\|_{\frac{q}{2}}D^{2}}\;
	\exp\Bigl(-\sqrt{C_s}\,C_6(n,p,q)\bigl(\|\lambda^{-}_{n,l}\|_{\frac{q}{2}}D^{2}\bigr)^{\frac{1}{4}}\Bigr),
	\end{align}
	and
	\begin{align}\label{eq:cond2}
	\sqrt{\|\operatorname{Ric}^{-}\|_{\frac{q}{2}}D^{2}}
	&\le \sqrt{2}\,\sqrt{\|\lambda^{-}_{n,l}\|_{\frac{q}{2}}D^{2}} \nonumber \\
	&\le \sqrt{2}\,\frac{1}{6\sqrt{2}\,e\,C_{8}(n,p)C_{1}(p,q)} \nonumber \\
	&\le \frac{1}{6\sqrt{e}\,C_s C_{1}(p,q)}
	\exp\Bigl(-\sqrt{C_s}\,C_6(n,p,q)\bigl(\|\lambda^{-}_{n,l}\|_{\frac{q}{2}}D^{2}\bigr)^{\frac{1}{4}}\Bigr).
	\end{align}
	Finally,
	\begin{align}\label{eq:cond3}
	\sqrt{\|\operatorname{Ric}^{-}\|_{\frac{q}{2}}D^{2}}
	&\le \sqrt{2}\,\sqrt{\|\lambda^{-}_{n,l}\|_{\frac{q}{2}}D^{2}} \nonumber \\
	&\le \sqrt{2}\,\frac{1}{2\sqrt{2}\,C_{8}(n,p)C_{4}(p,q)} \nonumber \\
	&= \frac{1}{2\,C_{8}(n,p)C_{4}(p,q)} = \frac{1}{2\,C_s\,C_{4}(p,q)}.
	\end{align}
Inequalities \eqref{eq:cond1}--\eqref{eq:cond3} verify that the three conditions of Theorem~\ref{thm:main1} are satisfied for the potential \(V = \operatorname{Ric}\).  
Consequently, by Theorem~\ref{thm:main1} there exists a scaling parameter \(t>0\) such that  

\[
\ker\mathcal{D}_{t\theta} \cap \Omega^{\pm}_{l}(M) = \{0\}.
\]
Because \(\theta\) is closed, we have the decomposition  
\[
\ker\mathcal{D}_{t\theta} \cap \Omega^{\bullet}_{l}(M)
= \bigoplus_{k \le l} \ker\mathcal{D}_{t\theta} \cap \Omega^{k}(M)
\bigoplus_{k \le l} \ker\mathcal{D}_{t\theta} \cap \Omega^{n-k}(M),
\]
where \(\bullet = +\) corresponds to even \(k\) and \(\bullet = -\) to odd \(k\).  

Finally, recalling that \(\ker\mathcal{D}_{t\theta} \cap \Omega^{k}(M) \cong H^{k}(M,t\theta)\) (the \(k\)-th Morse--Novikov cohomology group), we conclude the proof.
\end{proof}

\begin{proof}[\textbf{Proof of Theorem \ref{thm:main2}}]
	Since \(H^1_{\text{dR}}(M) \neq 0\), there exists a nonzero harmonic 1-form \(\theta\).  
	By Corollary \ref{cor:vanishing-kernel}, under the condition  
	\[
	\big\|\lambda_{2n,l}^{-}\big\|_{\frac{q}{2}}\, D^{2} \le C_{10},
	\]  
	we can choose \(t>0\) such that  
	\[
	H^{k}(M,\,t\theta)=0 \qquad\text{for all }k\le l,
	\]  
	where \(H^{k}(M,t\theta)\) denotes the Morse--Novikov cohomology twisted by \(t\theta\).
	
	Using Poincaré duality for Morse--Novikov cohomology \cite{Novikov1982},  we have
	\[
	\dim H^{k}(M,t\theta)=\dim H^{2n-k}(M,t\theta).
	\]  
	Hence the Euler characteristic can be written as  
	\begin{align*}
	\chi(M) 
	&= \sum_{k=0}^{2n}(-1)^{k}\dim H^{k}(M,t\theta) \\
	&= \sum_{k<n} 2(-1)^{k}\dim H^{k}(M,t\theta)
	+ (-1)^{n}\dim H^{n}(M,t\theta).
	\end{align*}
	If \(l=n\), the hypothesis gives \(H^{k}(M,t\theta)=0\) for all \(k\le n\).  
	By duality it follows that \(H^{k}(M,t\theta)=0\) for every \(k\), and consequently  
	\[
	\chi(M)=0.
	\]
	If \(l=n-1\), then \(H^{k}(M,t\theta)=0\) for all \(k\neq n\).  
	Thus  
	\[
	\chi(M)=(-1)^{n}\dim H^{n}(M,t\theta),
	\]  
	which yields  
	\[
	(-1)^{n}\chi(M)=\dim H^{n}(M,t\theta)\ge 0.
	\]
This completes the proof.
\end{proof}
\begin{corollary}\label{cor:eigenvalue-lower-bound}
	Let \((M^{2n}, g)\) be a closed \(2n\)-dimensional smooth Riemannian manifold with nonzero first de Rham cohomology group, and let $q>p>2n$. Suppose that the sectional curvature and diameter satisfy
	\[
	|\sec(g)|(x)\leq K(x), \qquad{\rm{diam}}(M)\leq D.
	\]
	where $K(x)$ is a nonnegative function on $M$. There exists a positive constant $C_{10}(n,p,q)$ such that if 
	$$\|K\|_{\frac{q}{2}}D^{2}\leq\frac{3}{2n(2n-1)}C_{10}(n,p,q),$$
then the Euler characteristic of $M$ satisfies $\chi(M)=0$.
\end{corollary}
\begin{proof}
	By \cite[Proposition 3.8]{Bourguignon1978}, the eigenvalues of the curvature operator obey
	\[
	\lambda_{1} + \dots + \lambda_{n} \ge -\frac{2n(2n-1)}{3}\,K.
	\]
	Consequently, if 
	\[
	\|K\|_{\frac{q}{2}}\, D^{2} \le \frac{3}{2n(2n-1)}\,C_{10},
	\]
	then 
	\[
	\big\|\lambda^{-}_{2n,n}\big\|_{\frac{q}{2}}\, D^{2} \le C_{10}.
	\]
	Applying Theorem~\ref{thm:main2} with \(l = n\) gives \(\chi(M) = 0\).
\end{proof}
\begin{proof}[\textbf{Proof of Corollary \ref{cor:4dim}}]
	Let \((M^4,g)\) be a closed 4‑dimensional Riemannian manifold. 
	The Euler characteristic can be expressed in terms of the Betti numbers as
	\[
	\chi(M) = 2 + b_2(M) - 2b_1(M).
	\]
If \(b_1(M)=0\), then \[\chi(M)=2+b_2(M)\ge 2>0.\]
Assume now that \(b_1(M)\ge 1\). Then there exists a non‑zero harmonic 1‑form \(\theta\).
	Applying Corollary \ref{cor:vanishing-kernel} with \(n=4\) and \(l=1\), 
	the hypothesis \(\|\operatorname{Ric}^{-}\|_{\frac{q}{2}}D^{2}\le C_{10}\) guarantees the existence of a number \(t>0\) such that
	\[
	H^{0}(M,t\theta)=H^{1}(M,t\theta)=0.
	\]
	Poincaré duality gives \(H^{3}(M,t\theta)=H^{4}(M,t\theta)=0\). Consequently,
	\[
	\chi(M)=\dim H^{2}(M,t\theta)\ge 0.
	\]
In both cases we obtain \(\chi(M)\ge 0\), which completes the proof.
\end{proof}

	\section{Eigenvalue Estimates for the Rough Laplacian}
	We now turn to eigenvalue estimates for the rough Laplacian on 1-forms under integral curvature conditions. First, we recall a Poincar\'e-type inequality on manifolds with almost nonnegative Ricci curvature.
	\begin{theorem}\label{thm:poincare-almost-nonnegative}
		Let \((M^{n},g)\) be a closed \(n\)-dimensional smooth Riemannian manifold satisfying
		\[
		\operatorname{Ric}(g) D^2(g) \ge -(n-1)b^2
		\]
		for some constant \(b > 0\). Then for each \(1 \le p \le \frac{nq}{n-q}\), \(p < \infty\), and \(f \in L^q_1(M)\), we have
		\[
		\|f - \bar{f} \|_{L^p(M)} 
		\le S_{p,q} \|df\|_{L^q(M)}, \label{eq:poincare-ineq} 
		\]
		where
		$$
		\bar{f}=\frac{1}{\operatorname{Vol}(g)} \int_M f d \operatorname{vol}_g,
		S_{p,q} = \Bigl(\frac{\operatorname{Vol}(g)}{\operatorname{Vol}(S^n(1))}\Bigr)^{\frac{1}{p}-\frac{1}{q}} 
		R(b) \Sigma(n,p,q),
		$$
		\(\Sigma(n,p,q)\) is the Sobolev constant of the canonical unit sphere \(S^n(1)\), and \(R(b) = \frac{D(g)}{b C(b)}\) with \(C(b)\) being the unique positive root of
		\begin{equation}\label{eq:Cb-equation}
		x \int_0^b (\cosh t + x \sinh t)^{n-1} dt = \int_0^\pi \sin^{n-1} t \, dt.
		\end{equation}
	\end{theorem}
	
	\begin{proof}
		See \cite[Theorem 5.1]{Chen2020} or \cite[Page 397]{Berard1988}.
	\end{proof}
	In order to use Theorem \ref{thm:poincare-almost-nonnegative}, the following estimate for the constant \(C(b)\) will be crucial for obtaining explicit lower bounds.
	
	\begin{lemma}(\cite[Lemma 2.3]{Huang2025b})\label{lem:Cb-estimate}
		Let \(C(b)\) be the unique positive root of equation \eqref{eq:Cb-equation}. Then
		\[
		b C(b) \ge a_n e^{-(n-1)b} > 0
		\]
		for some constant \(a_n\) depending only on \(n\).
	\end{lemma}
	
	The following proposition is a direct consequence of Theorem \ref{thm:poincare-almost-nonnegative} and Lemma \ref{lem:Cb-estimate}.

	\begin{proposition}\label{P1}
		Let $(M^{n},g)$ be a closed $n$-dimensional smooth Riemannian manifold. Denote by $\lambda_{1}\leq\cdots\leq\lambda_{{n\choose 2}}$ the eigenvalues of the curvature operator and set 
		$$\lambda_{n,l}:=\lambda_{1}+\cdots+\lambda_{n-l}.$$ 
		Fix $D>0$ and $1\leq l\leq\left\lfloor\frac{n}{2}\right\rfloor$. If $f \in L^2_1(M)$, then
		\[
		\|f-\bar{f}\|_{2} \leq C(n) D e^{\frac{(n-1)}{\sqrt{n-l}}\sqrt{\|\la^{-}_{n,l}\|_{\infty} D^{2}}}  \|df\|_{\frac{2n}{n+2}}.
		\]
	\end{proposition}
	\begin{proof}
Noting that
		\[
		\operatorname{Ric}\geq-\frac{n-1}{n-l}\la^{-}_{n,l}.
		\]
By Theorem \ref{thm:poincare-almost-nonnegative} and Lemma \ref{lem:Cb-estimate}  (with \(p=2\) and \(q=\frac{2n}{n+2}\), we may choose the Sobolev constant
		\[
		C_s = C(n) e^{\frac{(n-1)}{\sqrt{n-l}}\sqrt{\|\la^{-}_{n,l}\|_{\infty} D^{2}}},
		\]
		where \[C(n)=a^{-1}_{n}\Bigl(\operatorname{Vol}(S^n(1))\Bigr)^{\frac{1}{n}}\Sigma(n,2,\frac{2n}{n+2}). \]
	\end{proof}
We now prove the eigenvalue estimate.
\begin{proof}[\textbf{Proof of Theorem \ref{thm:eigenvalue}}]
	Let \(\theta\) be a nontrivial 1-form attaining the first eigenvalue:
	\[
	\nabla^*\nabla\theta-\lambda_1^{(1)}\theta =0 .
	\]
	This equation is a special case of the equation in Theorem \ref{thm:main1} with the potential \(V \equiv -\lambda_1^{(1)}\).
	Observe that
	\[
	\|V^{-}\|_{\infty}= \lambda_1^{(1)},
	\]
	and the Sobolev constant satisfies
	\[
	C_s(\infty) = C(n)\exp\Bigl(\frac{2n-1}{\sqrt{n}}
	\sqrt{\|\lambda^{-}_{2n,n}\|_{\infty}D^{2}}\Bigr).
	\]
	If, in addition,
	\begin{align*}
	\sqrt{\lambda_1^{(1)}D^{2}}\leq\min\bigg\{&e\sqrt{\|\la^{-}_{2n,n}\|_{\infty}D^{2}}\exp\Bigl(-\sqrt{C_s(\infty)}C_6(n,2n,\infty)(\|\la^{-}_{2n,n}\|_{\infty}D^{2})^{\frac{1}{4}}\Big),\\
	& \frac{1}{6\sqrt{e}C_s(\infty) C_{1}(n,\infty)}\exp\Bigl(-\sqrt{C_s(\infty)}C_6(n,2n,\infty)(\|\la^{-}_{2n,n}\|_{\infty}D^{2})^{\frac{1}{4}}\Bigr ),\\
	&\frac{1}{2C_s(\infty) C_{4}(n,\infty)}\bigg\},
	\end{align*}
	then by Theorem \ref{thm:main1} (with \(p=2n\) and \(q=+\infty\) in this case) we would obtain \(\chi(M)=0\), contradicting the hypothesis that \(\chi(M)\neq0\).
	Hence the opposite strict inequality must hold, which is precisely the claimed estimate.
\end{proof}
\begin{proposition}\label{prop:eigenvalue-Hodge-Laplacian}
	Let $(M^{2n},g)$ be a closed $2n$-dimensional, $n\ge 2$, smooth Riemannian manifold with nonvanishing Euler characteristic. Suppose that the Ricci curvature satisfies
	$$\operatorname{Ric}(g)\geq0.$$
	Then the first eigenvalue of the Hodge--Laplacian on $1$-forms satisfies
	\begin{equation*}
	\mu_{1}^{(1)}\geq\lambda^{(1)}_{1}.
	\end{equation*}
\end{proposition}
\begin{proof}
Since $\operatorname{Ric}(g) \ge 0$, for any $\alpha \in \Omega^1(M)$ we have
	\[
	\frac{(\Delta_d \alpha, \alpha)_{L^{2}}}{\|\alpha\|_{L^{2}}^{2}} \ge \frac{(\nabla^{*}\nabla \alpha, \alpha)_{L^{2}}}{\|\alpha\|_{L^{2}}^{2}}.
	\]
	Under the condition $\operatorname{Ric}(g) \ge 0$, any harmonic $1$-form $\alpha$ satisfies $\nabla\alpha = 0$, i.e., $\alpha$ is parallel. Because $\chi(M)\neq 0$, there are no non-zero harmonic $1$-forms. The first non-zero eigenvalue of the Hodge--Laplacian on $1$-forms is
	\[
	\mu^{(1)}_{1}= \inf_{\a\in\Om^{1}(M)\backslash\{0\}}\frac{(\Delta_{d}\alpha,\alpha)_{L^{2}(M)}}{\|\alpha\|^{2}_{L^{2}(M)}}.
	\]
	and the first eigenvalue of the rough Laplacian on $1$-forms is
	\[
	\lambda^{(1)}_{1}= \inf_{\a\in\Om^{1}(M)\backslash\{0\}}\frac{\|\nabla\alpha\|^{2}_{L^{2}(M)}}{\|\alpha\|^{2}_{L^{2}(M)}}.
	\]
We immediately conclude
	\[
	\mu^{(1)}_{1} \ge \lambda^{(1)}_{1}.
	\]
	
	\noindent This completes the proof.
\end{proof}

\begin{proof}[\textbf{Proof of Corollary \ref{cor:euler-betti-choice}}]
	We only need to prove that $\chi(M)=0$ when $b_{1}(M)>0$. Suppose, to the contrary, that $b_{1}(M)>0$ but $\chi(M)\neq0$. 
	Since $\chi(M)\neq0$, by Theorem \ref{thm:eigenvalue} the first eigenvalue of the rough Laplacian on $1$-forms satisfies \eqref{V8}. Because $\|\lambda^{-}_{2n,n}\|_{\infty}D^{2}\leq C_{11}(n)$ and we may choose $C_{11}(n)$ sufficiently small, the second and third terms inside the minimum is larger than the first. Consequently,
	\begin{equation}\label{eq:lower}
	\begin{split}
	\sqrt{\lambda_{1}^{(1)}D^{2}}&>e\sqrt{\|\la^{-}_{2n,n}\|_{\infty}D^{2}}\exp\Bigl(-\sqrt{C_s(\infty)}C_6(n,2n,\infty)(\|\la^{-}_{2n,n}\|_{\infty}D^{2})^{\frac{1}{4}}\Big)\\
	&\geq e^{\frac{1}{2}}\sqrt{\|\la^{-}_{2n,n}\|_{\infty}D^{2}}.
	\end{split}
	\end{equation}
	On the other hand, since $b_{1}(M)>0$, there exists a non-trivial harmonic $1$-form $\theta$. 
	The Weitzenb\"ock formula gives
	\[
	\nabla^{*}\nabla\theta+\operatorname{Ric}(\theta)=0.
	\]
	Taking the inner product with $\theta$ and integrating yields
	\[
	\|\nabla\theta\|_{2}^{2}= -\frac{1}{\operatorname{Vol}(g)}\int_{M}\langle\operatorname{Ric}(\theta),\theta\rangle
	\leq \|\operatorname{Ric}^{-}\|_{\infty}\|\theta\|_{2}^{2}.
	\]
	Using the pointwise estimate $\operatorname{Ric}^{-}\leq\frac{2n-1}{n}\lambda^{-}_{2n,n}$, we obtain
	\begin{equation}\label{eq:upper}
	\|\nabla\theta\|_{2}^{2}\leq \frac{2n-1}{n}\|\lambda^{-}_{2n,n}\|_{\infty}\|\theta\|_{2}^{2}.
	\end{equation}
	Noting that $\ker\na\cap\Om^{1}=0$ when $\chi(M)\neq0$. Hence, by the variational characterization of $\lambda_{1}^{(1)}$,
	\begin{equation*}
	\begin{split}
	\la_{1}^{(1)}:&=\inf_{\a\in\Om^{1}(M)\backslash\{0\} }\frac{\|\na\a\|_{2}^{2} }{\|\a\|_{2}^{2}}\leq\frac{\|\na\theta\|_{2}^{2} }{\|\theta\|_{2}^{2}}\\
	&\leq\|\operatorname{Ric}^{-}\|_{\infty}\\
	&\leq \frac{2n-1}{n}\|\la^{-}_{2n,n}\|_{\infty}.
	\end{split}
	\end{equation*}
	Combining this with \eqref{eq:lower} gives
	\[
e\|\lambda^{-}_{2n,n}\|_{\infty}
	<\lambda_{1}^{(1)}\leq\frac{2n-1}{n}\|\lambda^{-}_{2n,n}\|_{\infty},
	\]
	which is impossible because $e>2>\frac{2n-1}{n}$. 
	Therefore, when $b_{1}(M)>0$, we must have $\chi(M)=0$.
\end{proof}
\begin{proof}[\textbf{Proof of Theorem \ref{thm:eigenvalue-Hodge-Laplacian}}]
	By Corollary \ref{cor:euler-betti-choice} and the condition $\chi(M)\neq 0$, we have $b_{1}(M)=0$. Consequently,
	\[
	\mu^{(1)}_{1}= \inf_{\alpha \in \Omega^{1}(M)\setminus\{0\}} 
	\frac{(\Delta_{d}\alpha,\alpha)_{L^{2}(M)}}{\|\alpha\|^{2}_{L^{2}(M)}} .
	\]
	Moreover, Theorem \ref{thm:eigenvalue} yields (see \eqref{eq:lower})
	\[
	\lambda^{(1)}_{1} > e\|\lambda^{-}_{2n,n}\|_{\infty}.
	\]
	For any non-zero $1$-form $\alpha\in \Omega^{1}(M)$,
	\begin{align*}
	(\Delta_{d}\alpha,\alpha)_{L^{2}(M)}
	&= (\nabla^{*}\nabla\alpha,\alpha)_{L^{2}(M)} + (\operatorname{Ric}(\alpha),\alpha)_{L^{2}(M)} \\
	&\ge \lambda^{(1)}_{1}\,\|\alpha\|^{2}_{L^{2}(M)} 
	- \frac{2n-1}{n}\,\|\lambda^{-}_{2n,n}\|_{\infty}\,\|\alpha\|^{2}_{L^{2}} \\
	&> \left[ e\|\lambda^{-}_{2n,n}\|_{\infty} 
	- \frac{2n-1}{n}\,\|\lambda^{-}_{2n,n}\|_{\infty} \right] \|\alpha\|^{2}_{L^{2}} \\
	&> (e-2)\|\lambda^{-}_{2n,n}\|_{\infty}\,\|\alpha\|^{2}_{L^{2}}.
	\end{align*}
	This implies that for any non-zero $\alpha\in \Omega^{1}(M)$,
	\[
	\frac{(\Delta_{d}\alpha,\alpha)_{L^{2}(M)}}{\|\alpha\|^{2}_{L^{2}(M)}} 
	> (e-2)\|\lambda^{-}_{2n,n}\|_{\infty}.
	\]
	Taking the infimum over all non-zero $1$-forms $\alpha$ yields the desired lower bound for $\mu^{(1)}_{1}$.
\end{proof}
	\section*{Acknowledgements}
	This work is supported by the National Natural Science Foundation of China Nos. 12271496 (Huang) and the Youth Innovation Promotion Association CAS, the Fundamental Research Funds of the Central Universities, the USTC Research Funds of the Double First-Class Initiative. The authors also thank DeepSeek for its assistance in proofreading and improving the grammar and expression of this manuscript.
	
	\section*{Declarations}
	
	\subsection*{Conflict of Interest}
	The authors declare that there is no conflict of interest.
	
	\subsection*{Data Availability}
	This manuscript has no associated data.


\bigskip
	\footnotesize
	
	\begin{thebibliography}{99}
		
		\bibitem{Anderson1992}
		Anderson, M., \emph{Hausdorff perturbations of Ricci-flat manifolds and the splitting theorem}, Duke Math. J. \textbf{68} (1992), no. 1, 67--82.
		
		\bibitem{Aubry2007}
		Aubry, E.,  
		\textit{Finiteness of $\pi_{1}$ and geometric inequalities in almost positive Ricci curvature.} 
        Ann. Sci. École Norm. Sup. \textbf{40} (2007), no. 4, 675--695.
		
		\bibitem{Aubry2009}
		Aubry, E., \emph{Diameter pinching in almost positive Ricci curvature}, Comment. Math. Helv. \textbf{84} (2009), no. 2, 223--233.
		
		\bibitem{Aubry2003}
		Aubry, E., Colbois, B., Ghanaat, P., Ruh, E.A., \emph{Curvature, Harnack's inequality, and a spectral characterization of nilmanifolds}, Ann. Global Anal. Geom. \textbf{23} (2003), no. 3, 227--246.
		
		\bibitem{Anne2024}	
		Ann\'{e}, C., Takahashi, J.,
		\emph{Small eigenvalues of the rough and Hodge Laplacians under fixed volume.}
		Ann. Fac. Sci. Toulouse Math. (6) \textbf{33} (2024), no. 1, 123--151. 	
		
		\bibitem{Anne2025}	
		Ann\'{e}, C., Takahashi, J.,	
		\emph{Small eigenvalues of the Hodge-Laplacian with sectional curvature bounded below.}
		Ann. Global Anal. Geom. \textbf{68} (2025), no. 1, Paper No. 1, 15 pp. 
		
		\bibitem{Ballmann2003}
		Ballmann, W., Br\"{u}ning, J., Carron, G.,
		\emph{Eigenvalues and holonomy.} Int. Math. Res. Not. \textbf{12} (2003), 657--665.
		
		\bibitem{Berard1988}
		 Bérard, P.H., \emph{From vanishing theorems to estimating theorems: the Bochner technique revisited.} Bull. Amer. Math. Soc. \textbf{19} (1988), 371--406.
		
		\bibitem{Berger1961}
		Berger, M., \emph{Sur les variétés à opérateur de courbure positif.} Comptes Rendus Acad. Sci. Paris \textbf{253} (1961), 2832--2834.
		
		\bibitem{Bochner1946}
		Bochner, S.,  \emph{Vector fields and Ricci curvature.} Bull. Amer. Math. Soc. \textbf{52} (1946), 776--797.
		
		\bibitem{Boulanger2022}
		Boulanger, A., Courtois, G.,
		\textit{A Cheeger-like inequality for coexact $1$-forms.} 
		Duke Math. J. \textbf{171} (2022), no. 18, 3593--3641. 
		
		
		\bibitem{Bourguignon1978}
	Bourguignon, J.-P., Karcher, H.,  \emph{Curvature operators: pinching estimates and geometric examples.} Ann. Sci. École Norm. Sup. \textbf{11} (1978), 71--92.
		
		\bibitem{Buser1982}
		Buser, P.,  \emph{A note on the isoperimetric constant.} Ann. Sci. École Norm. Sup. \textbf{15} (1982), 213--230.
		
		\bibitem{Buser1981}
		Buser, P., Karcher, H.,  \emph{Gromov's almost flat manifolds.} Astérisque \textbf{81} (1981), 1--148.
		
		\bibitem{Cheeger1970}
		Cheeger, J.,
		\emph{ A lower bound for the smallest eigenvalue of the Laplacian.} Problems in analysis (Papers dedicated to Salomon Bochner, 1969), pp. 195--199. Princeton Univ. Press, Princeton, N. J., 1970.
			
		
		\bibitem{Chen2020}
	Chen, X.Y., \emph{Morse-Novikov cohomology of almost nonnegatively curved manifolds.} Adv. Math. \textbf{371} (2020), 107249, 23 pp.
		
	
		
		\bibitem{Chen2024b}
	Chen, X.Y., Han, F., \emph{New Bochner type theorems.} Math. Ann. \textbf{388} (2024), 3757--3783.
		
		\bibitem{Cheng1975}
		Cheng, S., \emph{Eigenvalue comparison theorems and its geometric applications.} Math. Z. \textbf{143} (1975), 289--297.
		
		\bibitem{Colbois1990}
		Colbois, B., Courtois, G.,
		\emph{A note on the first nonzero eigenvalue of the Laplacian acting on $p$-forms.}
		Manuscripta Math. \textbf{68} (1990), no. 2, 143--160. 
		
		\bibitem{Colbois2010}
		Colbois, B., Maerten, D.,
		\emph{Eigenvalue estimate for the rough Laplacian on differential forms.}
		Manuscripta Math. \textbf{132} (2010), 399--413. 
		
		
		\bibitem{deLeon2003}
		de Leon, M., Lopez, B., Marrero, J.C., Padron, E., \emph{On the computation of the Lichnerowicz-Jacobi cohomology.} J. Geom. Phys. \textbf{44} (2003), 507--522.
		
		\bibitem{Donaldson1983}
		Donaldson, S.K., \emph{An application of gauge theory to four-dimensional topology.} J. Differential Geom. \textbf{18} (1983), 279--315.
		
		\bibitem{Donaldson1987}
		Donaldson, S.K., \emph{The orientation of Yang-Mills moduli spaces and 4-manifold topology.} J. Differential Geom. \textbf{26} (1987), 397--428.
		

		
		\bibitem{Gallot1981}
		Gallot, S., \emph{Estimees de Sobolev quantitatives sur les variétés riemanniennes et applications.} C. R. Acad. Sci. Paris Sér. I Math. \textbf{292} (1981), 375--377.
		
		\bibitem{Gallot1988}
		Gallot, S., \emph{Isoperimetric inequalities based on integral norms of Ricci curvature.} Astérisque \textbf{157-158} (1988), 191--216.
		
		\bibitem{Gromov1978}
		Gromov, M., \emph{Almost flat manifolds.} J. Differential Geom. \textbf{13} (1978), 231--241.
		
		\bibitem{Gromov1981}
		Gromov, M., \emph{Curvature, diameter and Betti numbers.} Comment. Math. Helv. \textbf{56} (1981), 179--195.
		
		\bibitem{Herrmann2013}
		Herrmann, M., Sebastian D., Tuschmann, W., \emph{Manifolds with almost nonnegative curvature operator and principal bundles.} Ann. Glob. Anal. Geom. \textbf{44} (2013), 391--399.
		
		\bibitem{Honda2025}
		Honda, S., Mondino, A.,
		 \emph{Poincaré inequality for one-forms on four manifolds with bounded Ricci curvature.}
		 Arch. Math. \textbf{124} (2025), 449--455.
		
		\bibitem{Huang2025}
		Huang, T., Tan, Q., \emph{Curvature operator and Euler number.} Calc. Var. Partial Differential Equations \textbf{64} (2025), no. 7, Paper No. 205, 31 pp.
		
		\bibitem{Huang2025b}
		Huang, T., Wang, W., \emph{Eigenvalue Estimate for the Rough Laplacian on $1$-Forms and its Applications.} arXiv:2512.04740 (2025)
	
				
		\bibitem{Lawson1989}
		Lawson, H.B., Michelsohn, M.-L.,  \emph{Spin Geometry.} Princeton University Press, 1989.
		
		\bibitem{Li1980}
		Li, P., \emph{On the Sobolev constant and the $p$-spectrum of a compact Riemannian manifold.} Ann. Sci. École Norm. Sup. \textbf{13} (1980), 451--468.
		
		\bibitem{Li1980b}
		Li, P., Yau, S.T.,
		\emph{Estimates of eigenvalues of a compact Riemannian manifold.} Proc. Sympos. Pure Math. Am. Math. Soc. \textbf{36} (1980), 205--239.
		
		\bibitem{Lipnowski2018}
		Lipnowski, M., Stern, M.,
		\emph{Geometry of the smallest $1$-form Laplacian eigenvalue on hyperbolic manifolds.}
		Geom. Funct. Anal. \textbf{28} (2018), no. 6, 1717--1755. 
		
		\bibitem{Lichnerowicz1977}
		Lichnerowicz, A., \emph{Les variétés de Poisson et leurs algebres de Lie associees.} J. Differential Geom. \textbf{12} (1977), 253--300.
		
		\bibitem{Lott2004}
		Lott, J. \emph{Remark about the spectrum of the $p$-form Laplacian under a collapse with curvature bounded below.} Proc. Amer. Math. Soc. \textbf{132} (2004), 911--918.
		
		\bibitem{Meyer1971}
		Meyer, D., \emph{Sur les variétés riemanniennes à opérateur de courbure positif.} C. R. Acad. Sci. Paris Sér. A-B \textbf{272} (1971), A482--A485.
		
		\bibitem{Novikov1982}
		 Novikov, S.P., \emph{The Hamiltonian formalism and a multivalued analogue of Morse theory.} Uspekhi Mat. Nauk \textbf{37} (1982), 3--49.
		
		\bibitem{Petersen1998}
		Petersen, P., \emph{Riemannian Geometry.} Springer-Verlag, 1998.
		
		\bibitem{Petersen1997a}
		Petersen, P., Shteingold, S.D.,  Wei, G.,  \emph{Comparison geometry with integral curvature bounds.} Geom. Funct. Anal. \textbf{7} (1997), no. 6, 1011--1030.
		
		\bibitem{Petersen1997b}
		Petersen, P., Wei, G., \emph{Relative volume comparison with integral curvature bounds.} Geom. Funct. Anal. \textbf{7} (1997), no. 6, 1031--1045.
		
		\bibitem{Petersen1998}
	Petersen, P., Sprouse, C., \emph{Integral curvature bounds, distance estimates and applications.} J. Differential Geom. \textbf{50} (1998), no. 2, 269--298.
		
		\bibitem{Petersen2021}
		Petersen, P., Wink, M.,  \emph{New curvature conditions for the Bochner technique.} Invent. Math. \textbf{224} (2021), 33--54.
		
		\bibitem{Ruh1982}
		Ruh, E.A.,  \emph{Almost flat manifolds.} J. Differential Geom. \textbf{17} (1982), 1--14.
		
		\bibitem{Schoen1994}
		Schoen, R., Yau, S.-T.,  \emph{Lectures on Differential Geometry.} International Press, Cambridge, MA, 1994.
		
		\bibitem{Shiohama1983}
		Shiohama, K., \emph{A sphere theorem for manifolds of positive Ricci curvature.} Trans. Amer. Math. Soc. \textbf{275} (1983), 811--819.
		
		\bibitem{Witten1994}
		Witten, E., \emph{Monopoles and four-manifolds.} Math. Res. Lett. \textbf{1} (1994), 769--796.
		
		\bibitem{Yamaguchi1991}
		Yamaguchi, T., \emph{Collapsing and pinching under a lower curvature bound.} Ann. of Math. \textbf{133} (1991), 317--357.
		
		\bibitem{Yau1975}
		Yau, S.-T., \emph{Harmonic functions on complete Riemannian manifolds.} Comm. Pure Appl. Math. \textbf{28} (1975), 201--228.
		
		\bibitem{Yau1982}
		Yau, S.-T.,
		\emph{Problem section. Seminar on Differential Geometry.} 
		Ann. of Math. Stud., No. \textbf{102}, Princeton Univ. Press, Princeton, NJ, 1982, pp. 669–706. 
		
		
		\bibitem{Yu2022}
		Yu, R. , \emph{Integral curvature bounds and Betti numbers.} (2022) arXiv:2211.05176v1.
		
		\bibitem{Zhong1984}
	Zhong, J.Q., Yang, H.C., \emph{On the estimate of the first eigenvalue of a compact Riemannian manifold.} Sci. Sinica Ser. A \textbf{27} (1984), 1265--1273.
		
	\end{thebibliography}
\end{document}